\documentclass[11pt]{amsart}

\usepackage{amsfonts,epsfig}
\usepackage{latexsym}
\usepackage{amssymb}
\usepackage{amsmath}
\usepackage{amsthm}
\usepackage{graphics}
\usepackage{verbatim}
\usepackage{enumerate}
\usepackage[all]{xy}
\usepackage{multirow}
\usepackage{array,booktabs,ctable}
\usepackage[backref]{hyperref}
\usepackage{color}
\definecolor{mylinkcolor}{rgb}{0.8,0,0}
\definecolor{myurlcolor}{rgb}{0,0,0.8}
\definecolor{mycitecolor}{rgb}{0,0,0.8}
\hypersetup{colorlinks=true,urlcolor=myurlcolor,citecolor=mycitecolor,linkcolor=mylinkcolor,linktoc=page,breaklinks=true}
\usepackage{bigstrut}
\usepackage{caption}

\newtheorem{defn}{Definition}[section]
\newtheorem{definition}[defn]{Definition}

\newtheorem{corollary}[defn]{Corollary}
\newtheorem{lemma}[defn]{Lemma}

\newtheorem{thm}[defn]{Theorem}
\newtheorem{theorem}[defn]{Theorem}

\newtheorem{prop}[defn]{Proposition}
\newtheorem{proposition}[defn]{Proposition}

\newtheorem{conj}[defn]{Conjecture}

\theoremstyle{definition}

\newtheorem{remark}[defn]{Remark}

\newtheorem{example}[defn]{Example}

\newcommand{\CC}{\mathbb C}

\newcommand{\QQ}{\mathbb Q}
\newcommand{\Q}{\mathbb Q}

\newcommand{\ZZ}{\mathbb Z}
\newcommand{\Z}{\mathbb Z}

\newcommand{\FF}{\mathbb F}

\newcommand{\PP}{\mathbb P}

\newcommand{\SL}{\operatorname{SL}}
\newcommand{\PSL}{\operatorname{PSL}}

\newcommand{\Ker}{\operatorname{Ker}}

\newcommand{\Gal}{\operatorname{Gal}}

\newcommand{\GL}{\operatorname{GL}}
\newcommand{\PGL}{\operatorname{PGL}}

\renewcommand{\Im}{\operatorname{Im}}

\newcommand{\im}{\mathrm{im}\,}

\newcommand{\Syl}{\operatorname{Syl}}

\newcommand{\disc}{\operatorname{disc}}

\begin{document}

% Title, authors and addresses

% use the thanksref command within \title, \author or \address for footnotes;
% use the corauthref command within \author for corresponding author footnotes;
% use the ead command for the email address,
% and the form \ead[url] for the home page:
% \title{Title\thanksref{label1}}
% \thanks[label1]{}
% \author{Name\corauthref{cor1}\thanksref{label2}}
% \ead{email address}
% \ead[url]{home page}
% \thanks[label2]{}
% \corauth[cor1]{}
% \address{Address\thanksref{label3}}
% \thanks[label3]{}

%\bibliographystyle{plain}
\title{Near Coincidences and Nilpotent Division Fields}

\author{Harris B. Daniels}
\address{Department of Mathematics and Statistics, Amherst College, Amherst, MA 01002, USA}
\email{hdaniels@amherst.edu}
\urladdr{http://www3.amherst.edu/~hdaniels/}

\author{Jeremy Rouse}
\address{Department of Mathematics, Wake Forest University, Winston-Salem, NC 27109, USA}
\email{rouseja@wfu.edu}
\urladdr{https://users.wfu.edu/rouseja/}

\keywords{Elliptic curves, Division fields, Galois groups, Modular curves}
{}
\subjclass[2020]{Primary: 11G05, Secondary: 11R32, 14H52.}

\begin{abstract}
Let $E/\QQ$ be an elliptic curve. We say that $E$ has a near coincidence of level $(n,m)$ if $m \mid n$ and $\QQ(E[n]) = \QQ(E[m],\zeta_{n})$. We classify near coincidences of prime power level and use this result to give a classification of values of $n$ for which $\Gal(\QQ(E[n])/\QQ)$ is a nilpotent group. Along the way we prove a Gauss-Wantzel analog for the elliptic curve $E\colon y^2 = x^3-x$, showing that $\QQ(E[n])/\QQ$ is constructible if and only if $\varphi(n)$ is a power of 2. Assuming that there are no non-CM rational points on the modular curves $X_{ns}^{+}(p)$ for primes $p > 11$, we show that $\Gal(\QQ(E[n])/\QQ)$ nilpotent implies that $n$ is a power of $2$ or $n \in \{ 3, 5, 6, 7, 15, 21 \}$.
\end{abstract}

\maketitle
\section{Introduction}
For much of mathematical history, compass and straightedge constructions have served as exemplars of mathematical reasoning. 
%Grounded in pure Euclidean geometry, these constructions exemplified a philosophy that prized deductive rigor over most everything else. 
At the center of the study of compass and straightedge constructions is the question of which regular polygons are constructible. First considered by the Greeks, an answer to this question remained elusive to mathematicians for hundreds of years.
%Euclid's {\it Elements} \cite[Book IV]{EuclidElements} provided constructions for the equilateral triangle, the square, the regular pentagon, decagon, and 15-gon, but there was no general theory for an arbitrary $n$-gon. In particular, there was no known construction of a 7- or 9-gon, and hundreds of years later, Ptolemy was still attempting to construct the 7-gon \cite[Book I]{PtolemyAlmagest}.
For centuries after this question was first considered, no progress was made. This changed when Gauss proved the following theorem:

\begin{thm}\cite{Gauss1801}
Suppose that $n$ is of the form 
\[
	n = 2^k\cdot p_1\cdots p_j
\] 
where $p_i$ is a Fermat prime. Then it is possible to construct a regular $n$-gon using only a straightedge and compass. 
\end{thm}
Here we remind the reader that a Fermat prime is a Fermat number (i.e., a number of the form $2^{2^{m}} + 1$ with $m \geq 0$) which is prime.

In fact, in \cite{Gauss1801} Gauss states that he has a proof of the converse, but he says, ``the limits of the present work exclude this demonstration here.'' The final part of the question was formally answered in 1837 when Pierre Wantzel proved the converse. 

\begin{thm}\cite{Wantzel}
Suppose $n$ is a positive integer such that a regular $n$-gon is constructible. Then it must be that $n$ is of the form 
\[
	n = 2^k\cdot p_1\cdots p_j
\] 
where each $p_i$ is a Fermat prime.
\end{thm}

Together, these two results together form the Gauss--Wantzel theorem. If we let $\varphi$ be Euler's totient function, then we can summarize the Gauss--Wantzel theorem as saying a regular $n$-gon is constructible if and only if there exists a $k$ such that $\varphi(n) = 2^k$.

A modern reformulation of this is to take the standard plane and turn it into the complex plane. A point in the complex plane is said to be {\it \bfseries constructible} if there is a way to define that point using just a straightedge and compass. Interpreting the Gauss--Wantzel theorem through this more modern lens gives us a new interpretation of what it means for a complex number to be constructible. 

\begin{thm}\cite[Theorem 4.53]{Rotman}\label{thm:ConstField}
Let $\alpha \in \CC$. Then, $\alpha$ is constructible if and only if there exists an ascending chain of fields
\[
	\QQ=K_0\subseteq K_1 \subseteq K_2 \subset \dots \subseteq K_j
\]
such that $\alpha\in K_j$ and $[K_{i}\colon K_{i-1}]=2.$
\end{thm}

A subfield of $\CC$ is said to be {\it \bfseries constructible} if all of the elements in the field are constructible. Using Galois theory from here we can update Theorem \ref{thm:ConstField} to say that $\alpha$ is constructible if and only if there is a Galois extension $K/\QQ$ that contains $\alpha$ such that $\Gal(K/\QQ)$ is a 2-group. 
Given that the $n$th roots of unity are equally spaced around the unit circle, from this updated perspective, the question of about the constructability of a regular $n$-gon is exactly the same as the question about the constructability of $\QQ(\zeta_n)$. Here $\zeta_n$ is a primitive $n$th root of unity.
This modern perspective  allows us to examine the constructability of more complicated fields. The primary objects of study in this paper are the division fields of elliptic curves. 

First we fix $\overline{\QQ}$, an algebraic closure of $\QQ$ and let $E/\QQ$ be an elliptic curve. A classical result is that the points on $E$ can be given the structure of an abelian group. We let $E[n]$ be the $n$-torsion on $E$ defined over $\overline{\QQ}$, that is $E[n] = \{P\in E(\overline{Q}) \colon n\cdot P = \mathcal{O} \}$. Then we can define $\QQ(E[n])$ to be the field of definition of the $n$-torsion points. With this we present the following theorem.

\begin{thm}{\rm (A Gauss-Wantzel Analog)}
Let $E$ be the elliptic curve given by $y^2 = x^3-x$ and let $n\geq 2$. Then, $\QQ(E[n])$ is a constructible field if and only if $\varphi(n) =2^k$ for some integer $k$.
\end{thm}

\begin{proof}
Suppose $n\in\ZZ$ such that $\QQ(E[n])/\QQ$ is a constructible extension. By properties of the Weil pairing, $\QQ(\zeta_n)\subseteq \QQ(E[n])$, and so by the Gauss--Wantzel Theorem it must be that $\varphi(n)$ is a power of 2. 

In the opposite direction, suppose that $n\in \ZZ$ such that $\varphi(n)$ is a power of 2. Given that constructability is closed under compositum, it is enough to understand when $\QQ(E[p^k])$ is constructible for some prime power $p^{k}$. 

Suppose $p=2$. In this case, $\QQ(E[2]) = \QQ$ 
and $[\QQ(E[2^{k+1}]):\QQ(E[2^{k}])] = 2^m$ 
for some $m\in\ZZ$ depending on $k$. Thus, for any $k$, $\Gal(\QQ(E[2^k])/\QQ)$ is a 2-group and hence constructible. 

Next suppose that $p = 2^{2^{n}} + 1$ is a Fermat prime and $E : y^{2} = x^{3} - x$. 
%Here is a self-contained proof that $[\Q(E[p]) : \Q]$ is a power of $2$.
We start this case by showing that $[\Q(E[p],i) : \Q(i)]$ is a power of $2$, which implies the desired result. 
The endomorphism ring of $E$ over $\Q(i)$ is isomorphic to $\Z[i]$. 
Moreover, since $p \equiv 1 \pmod{4}$ splits we can factor $p = \pi \overline{\pi}$ for $\pi \in \Z[i]$. 
(Explicitly, we have $\pi = 2^{2^{n-1}} + i$ and $\overline{\pi} = 2^{2^{n-1}} - i$.)
We therefore have two isogenies $\phi_{1} \colon E \to E$ and $\phi_{2} \colon E \to E$ given by $\phi_{1}(P) = [\pi] P$ and $\phi_{2}(P) = [\overline{\pi}]P$ which are defined over $\Q(i)$. Let $E[\pi]$ and $E[\overline{\pi}]$ be the kernels of these two isogenies. 
It is straightforward to see that $E[p]$ is the direct product of $E[\pi]$ and $E[\overline{\pi}]$, that each of $E[\pi]$ and $E[\overline{\pi}]$ has order $p$, and both $E[\pi]$ and $E[\overline{\pi}]$ are Galois stable over $\QQ(i)$. 
%If we choose generators $A$ and $B$ of $E[\pi]$ and $E[\overline{\pi}]$, it follows that
%for all $\sigma \in \Gal(\Q(E[p],i)/\Q(i))$, that $\phi_{1}(\sigma(A)) = \sigma(\phi_{1}(A)) = \sigma(0) = 0$ and so $\sigma(A) \in E[\pi]$.
%Likewise $\sigma(B) \in E[\overline{\pi}]$. 
This shows that the image of the mod $p$ Galois representation
$\rho_{E,p} \colon \Gal(\Q(E[p],i)/\Q(i)) \to \GL_{2}(\mathbb{F}_{p})$ is conjugate to a subgroup of the diagonal matrices. 
The set of diagonal matrices
has order $(p-1)^{2} = 2^{2^{n+1}}$. 
Since $\rho$ is injective, Lagrange's theorem implies that $[\Q(E[p],i) : \Q(i)]$ is a power of $2$.
\end{proof}

One might think that the more general question of when $\Gal(\QQ(E[n])/\QQ)$ can be a $p$-group would add some interest, but it turns out the only interesting case here is when $p=2$. 
This is easy to see since $[\QQ(E[n]):\QQ]$ is even for all $n\geq 3$.
If we want to study a condition on $\Gal(\QQ(E[n])/\QQ)$ that is more general than being a $p$-group, but is restrictive enough to yield meaningful results, a naive thing to do would be to study when $\Gal(\QQ(E[n])/\QQ)$ is the direct product of a finite number of $p$-groups. 
This allows some additional flexibility, but still imposes some structure on the situation. 

While at first this seems arbitrary and naive, this is a very natural thing to do.
\begin{prop}\cite[Proposition II.7.5]{hungerford}
Let $G$ be a finite group, then $G$ is nilpotent if and only if it is the direct product of its $p$-Sylow subgroups. 
\end{prop}
Thus, our naive consideration is actually logically equivalent to the condition that $\Gal(\QQ(E[n])/\QQ)$ is nilpotent. Looking at the literature, we find the following result about when $\Gal(\QQ(E[n])/\QQ)$ is abelian. 

\begin{thm}\cite[Theorem 1.1]{AbelDivFields}
Let $E/\QQ$ be an elliptic curve and $n\geq 2$. If $\QQ(E[n])/\QQ$ is an abelian extension, then $n=2,3,4,5,6,$ or $8$.
\end{thm}

With this added context, our new goal is to understand when $\Gal(\QQ(E[n])/\QQ)$ is nilpotent for $E/\QQ$. To that end, we give the following definition:
\begin{definition}
	Let $K/k$ be a Galois extension of fields. We say that $K/k$ is a {\it \bfseries nilpotent extension} if $\Gal({K}/k)$ is a nilpotent group. When the base field $k$ is clear from context, we will just say that $K$ is a {\bfseries nilpotent field} for brevity.
\end{definition}

We will be able to give a complete answer to the question of when $\QQ(E[n])/\QQ$ is nilpotent that is conditional on a positive answer to Serre's uniformity question.

\begin{conj}\label{conj:Uniformity}{\rm \cite[Conjecture 1.1]{ZywSurj}, \cite[Conjecture 1.1.5]{RSZB}}
If $p > 11$, then there is no non-CM elliptic curve $E/\Q$ for which the image of the mod $p$ Galois representation is contained in the normalizer of the non-split Cartan subgroup.
\end{conj}

Before stating our next theorem we remind the reader that a Mersenne prime is a prime which is one less than a power of $2$. Such a prime always has the form $2^{p} - 1$ for $p$ a prime. 

%Version 3

%\jeremy{I'm taking a stab at writing an unconditional version.}

\begin{theorem}
\label{thm:main_Nilp}
Let $E/\QQ$ be an elliptic curve.
\begin{enumerate}
\item If $E$ does not have complex multiplication, $\QQ(E[n])/\QQ$ is nilpotent, and Conjecture~\ref{conj:Uniformity} holds, then $n \in \{ 3, 5, 6, 7, 15, 21 \} \cup \{ 2^{k} \colon k \in \ZZ^{+} \}$. Each of these cases occurs for infinitely many different rational $j$-invariants.
\item If $E$ does not have complex multiplication, $n \not\in \{ 3, 5, 6, 7, 15, 21 \} \cup \{ 2^{k} \colon k \in \ZZ^{+} \}$, then $\QQ(E[n])/\QQ$ is nilpotent if and only if one of the following holds: (i) $n$ is a product of distinct Mersenne primes with the property that the mod $p$ image of Galois is contained in the normalizer of the non-split Cartan for all primes $p \mid n$, or (ii) $n$ is twice a product of distinct Mersenne primes with the property that the mod $p$ image of Galois is contained in the normalizer of the non-split Cartan, and the mod $2$ image has RSZB label {\tt 2.2.0.1}. 
\item If $E$ has complex multiplication by the order of discriminant $D \in \{ -4, -7, -8, -12, -16, -28 \}$, then $\QQ(E[n])/\QQ$ is nilpotent if and only if $n$ is a power of two times a product of distinct Mersenne and Fermat primes, where the Mersenne primes are inert in the CM field and the Fermat primes are split in the CM field.
\item If $E$ has complex multiplication by the order of discriminant $D \in \{ -11, -19, -43, -67, -163 \}$, then $\QQ(E[n])/\QQ$ is nilpotent if and only if $n$ is a product of distinct Mersenne and Fermat primes, where the Mersenne primes are inert in the CM field and the Fermat primes are split in the CM field.
\item If $E$ has complex multiplication by the order of discriminant $D = -27$, then $\QQ(E[n])/\QQ$ is never nilpotent.

\item If $j(E) = 0$, then $E$ is isomorphic over $\QQ$ to an elliptic curve of the form $E_d\colon y^2= x^3+d.$ Then, $\QQ(E_d[n])/\QQ$ is nilpotent if and only if $n = p$ is a prime and
\[
\begin{cases}
d \equiv 1 \pmod{(\QQ^{\times})^{3}} \text{ if } p = 2,\\
d \equiv 2 \pmod{(\QQ^{\times})^{3}} \text{ if } p = 3,\\
d \equiv 2\cdot p^{\frac{p-1}{3}} \pmod{(\QQ^{\times})^{3}} \text{ if }
p = 3 \cdot 2^{k} + 1\hbox{ for some $k\geq 1$},\\
d \equiv 2\cdot p^{\frac{p+1}{3}} \pmod{(\QQ^{\times})^{3}} \text{ if }
p = 3 \cdot 2^{k} - 1\hbox{ for some $k\geq 1$}.\\
\end{cases}
\]
\end{enumerate}
\end{theorem}

One corollary of this result is the following.

\begin{corollary}
\label{cor:no19}
There is no elliptic curve $E/\Q$ for which $\Q(E[19])/\Q$ is a nilpotent extension. Further, $19$ is the smallest prime with this property.
\end{corollary}

Along the way to proving Theorem~\ref{thm:main_Nilp}, a phenomenon appears that we point out here. When studying the question of when $\QQ(E[p^n])/\QQ$ nilpotent given that $\QQ(E[p^{n-1}])/\QQ$ is, we realized that we would need to classify exactly when $\QQ(E[p^{n}]) = \QQ(E[p^{n-1}],\zeta_{p^{n}})$. 

\begin{definition}
Let $E/\QQ$ be an elliptic curve and let $m,n\in \ZZ^+$ such that $m\mid n$. We say that there is a {\it{\bfseries near coincidence}} between the $m$- and $n$-division fields of $E$ if 
$$\QQ(E[n]) = \QQ(E[m],\zeta_n).$$
\end{definition}

The terminology here is motivated by the definition given in \cite{Coincidences} where an elliptic curve $E/\QQ$ is said to have a {\it{\bfseries coincidence}} between its division fields if $\QQ(E[m]) = \QQ(E[n])$ for distinct integers $m$ and $n$. In that paper is the following theorem:

\begin{theorem}\label{thm:MainCoin}{\rm \cite[Theorem 1.4]{Coincidences}}
Let $E/\QQ$ be an elliptic curve, $p$ be a prime, and let $n\in\ZZ^+.$
\begin{enumerate}
\item Suppose $\QQ(E[ p^{n+1} ]) = \QQ(E[p^{n}]).$ Then $p=2$ and $n=1.$
\item If $\QQ(E[p^n])\cap \QQ(\zeta_{p^{n+1}}) = \QQ(\zeta_{p^{n+1}}),$ then $p=2$.
\end{enumerate}
\end{theorem}

The problem of classifying such coincidences for elliptic curves defined over number fields was also recently investigated by Yvon in \cite{CoinOverNumberFields}. 

In the end, we are able to give a complete classification of elliptic curves over $\QQ$ with a near coincidence of their $p^{n+1}$ and $p^n$ division fields; i.e., when $\QQ(E[p^{n+1}])= \QQ(E[p^n],\zeta_p^{n+1}).$

\begin{thm}\label{thm:Main_NearCoin}
Let $E/\QQ$ be an elliptic curve, $p\in \ZZ$ prime, and $n\in \ZZ^+.$ Suppose that $$\QQ(E[p^{n+1}]) = \QQ(E[p^n],\zeta_{p^{n+1}}).$$
Then $p \in \{ 2, 3 \}$ and $n = 1$. Further, if $p=2$, then $E$ must correspond to a rational point on the one of the modular curves with RSZB label \texttt{4.48.0.3} or \texttt{4.16.0.2}, while if $p=3$, then $E$ must come from a rational point on \texttt{9.27.0.1}.
%4.48.0.3: Q(E[2]) = Q and Q(E[4]) = Q(i), and 
%4.16.0.2: Q(E[2]) = Q(E[4]);
%9.27.0.1: Q(E[9]) = Q(E[3],zeta_9)
\end{thm}

\begin{remark}
Before moving on, we point out that elliptic curves corresponding to rational points on \texttt{4.16.0.2} have the property that $\QQ(E[4]) = \QQ(E[2]) = \QQ(E[2],i).$ Thus, these curves actually have a coincidence of division fields. In contrast, the elliptic curves corresponding to rational points on \texttt{4.48.0.3} have the property that $\QQ(E[2]) = \QQ$ and $\QQ(E[4]) = \QQ(i)$. So these curves have a near coincidence without having a coincidence of division fields. It is interesting to note that in order to have a near coincidence between the $2$- and $4$-division fields, without having an actual coincidence between the $2$- and $4$-division fields, the $2$-division field has to be trivial.

Similarly, the rational points on \texttt{9.27.0.1} yield elliptic curves with a near coincidence between their 3- and 9-division fields. Unlike the case with $p = 2$, a rational point on \texttt{9.27.0.1} corresponds to an elliptic curve $E$
with either $j(E) = 0$ or a surjective $\bar\rho_{E,3}$. In particular,
if $E$ is a non-CM elliptic curve with $\QQ(E[9]) = \QQ(E[3],\zeta_{9})$, we must
have $[\QQ(E[3]) : \QQ] = 48$.
\end{remark}

\begin{remark} We choose to begin this paper with the Gauss--Wantzel theorem not only because of its classical significance, but because it serves as a natural entry point for a broad audience into the kinds of questions that we hope to address in this paper. The main objective here is to classify when the division field $\QQ(E[n])/\QQ$ of an elliptic curve $E$ over $\QQ$ is nilpotent. Along the way, we uncover a surprising analog of the Gauss--Wantzel theorem as well as a result about near coincidences of division fields. We hope that the Gauss--Wantzel analog, while not the main thrust of the paper, illustrates the value of considering questions like these.
\end{remark}

\subsection{Outline of the paper} In Section \ref{sec:Background} we recall basic facts related to elliptic and modular curves that will be necessary for the proof of the main results. The proof of Theorem \ref{thm:Main_NearCoin} will be handled in Section~\ref{sec:near} by finding the smallest power $n$ such that $\QQ(E[p^{n+1}])$ cannot be $\QQ(E[p^n],\zeta_{p^{n+1}})$ and then we prove that if $\QQ(E[p^{n+1}])\neq\QQ(E[p^n],\zeta_{p^{n+1}})$, then $\QQ(E[p^{n+2}])\neq\QQ(E[p^{n+1}],\zeta_{p^{n+2}})$. The proof will have to be broken down into cases depending on if $p=2, 3, 5,$ or $p\geq 7.$

The proof of Theorem \ref{thm:main_Nilp} starts in Section \ref{sec:PrimeLevel} by using group theory to classify the nilpotent subgroups of $\GL_2(\ZZ/p\ZZ)$ according to their image in $\PGL_2(\ZZ/p\ZZ)$. 
We then apply known results about the corresponding modular curves to determine when the extension $\QQ(E[p])/\QQ$ can be nilpotent. This is where we will employ Conjecture \ref{conj:Uniformity}.

In Section \ref{sec:PrimePowerLevel} we use the information from the previous section to prove that if $p$ is odd, then the $p^2$-division field is never nilpotent. Lastly, in Section \ref{sec:CompositeLevel} we study which combinations of mod $p$ images can occur simultaneously, and prove Theorem~\ref{thm:main_Nilp}. 
Throughout the paper we address the case of elliptic curves with complex multiplication separately from
those without complex multiplication because of their unique properties (which are outlined in Section \ref{subsec:CMEllipticCurves}).

All of the computations in this paper were performed using Magma \cite{Magma} and the code can be found at \cite{code}.

\subsection{Acknowledgements} We would like to thank David Zureick-Brown and \'Alvaro Lozano-Robledo for insightful conversations during the writing of this paper, and we would like to thank Pedro Lemos for helpful communications about the results in \cite{LemosSomeCases}. We would also like to thank the anonymous referees whose careful reading has helped improve the clarity and accuracy of the paper.

\section{Background}\label{sec:Background}
The goal of this section is to review some of the background information necessary for the proofs of the main theorems.
In each subsection readers should find some additional resources to supplement what is written here.

\subsection{Elliptic Curves}
For background about elliptic curves, see \cite{Silv1}. 
Given an elliptic curve $E/\QQ$ and a natural number $n$, the points of order dividing $n$ defined over $\overline{\QQ}$ form a group. Considering $E(\CC)$ as the quotient of $\CC$ by a lattice shows that
$$E[n] := \{P\in E(\overline{\QQ}) \colon nP =\mathcal{O} \}\simeq \ZZ/n\ZZ\oplus\ZZ/n\ZZ.$$
This isomorphism is non-canonical, but it only requires a choice of basis for $E[n]$. 

Because the group law on an elliptic curve is given by rational functions, $\Gal(\overline{\QQ}/\QQ)$ acts on $E[n]$ component-wise. 
That is, if $P = (x,y)\in E[n]$ and $\sigma\in \Gal(\overline{\QQ}/\QQ)$, then $P^\sigma = (\sigma(x),\sigma(y)) \in E[n]$.
This component-wise action induces a representation 
\[\
\bar\rho_{E,n}\colon \Gal(\overline{\QQ}/\QQ)\to \GL_2(\ZZ/n\ZZ)
\]
with the property that $\Im\bar\rho_{E,n} \simeq \Gal(\QQ(E[n])/\QQ).$ 
We remark here that because the isomorphism $E[n] \simeq \ZZ/n\ZZ\oplus\ZZ/n\ZZ$ is non-canonical, $\Im\bar\rho_{E,n}$ is really only defined up to conjugation. 
%This will not be a major point throughout the paper, but is worth pointing out. 

A guiding principle in this paper is that oftentimes things can be broken down into cases depending on the shape of $\Im\bar\rho_{E,p}.$ We are able to do this thanks to the following proposition. 

% \begin{prop}\label{prop:Serre_Max_Images}{\rm \cite{Serre72}}
% Let $E/\QQ$ be an elliptic curve and let $p$ be a prime. Let $G$ be the image of $\overline \rho_{E,p}\colon \Gal(\overline\QQ/\QQ)\to \Aut(E[p])\simeq \GL_2(\ZZ/p\ZZ)$. Then, there is exists a $\ZZ/p\ZZ$-basis for $E[p]$ such that one of the following is true:
% \begin{enumerate}
% \item[\rm (1)] $G = \GL_2(\ZZ/p\ZZ)$;
% \item[\rm (2)] $G$ is contained in a Borel subgroup of $\GL_2(\ZZ/p\ZZ)$;
% \item[\rm (3)] $G$ is contained in the normalizer of a split Cartan subgroup of $\GL_2(\ZZ/p\ZZ)$;
% \item[\rm (4)] $G$ is contained in the normalizer of a non-split Cartan subgroup of $\GL_2(\ZZ/p\ZZ)$;
% \item[\rm (5)] the image of $G$ in $\PGL_{2}(\ZZ/p\ZZ)$ is isomorphic to $S_{4}$.
% \end{enumerate}
% \end{prop}
% \jeremy{I'm slightly weirded out by this reference. First, there's no specific placement for this result in the paper. Also, this is basically group theory, and it's quite old, as Swinnerton-Dyer comments in the reference for the next result. Also, the next result is more or less a rewording of this one.}

\begin{proposition}{\rm \cite[Lemma 2]{Swinnerton-Dyer}}\label{prop:ProjIm}
	Let $p$ be a prime and let $G$ be a subgroup of $\GL_2(\ZZ/p\ZZ)$. 
	If $p$ divides $|G|$ then, either $\SL_2(\ZZ/p\ZZ)\subseteq G,$ or $G$ is contained inside a Borel subgroup of $\GL_2(\ZZ/p\ZZ)$. If $p$ does not divide $|G|$, let $H$ be the image of $G$ in $\PGL_2(\ZZ/p\ZZ)$; then
	\begin{enumerate}
		\item[\rm (1)] $H$ is cyclic and $G$ is contained inside a Cartan subgroup of $\GL_2(\ZZ/p\ZZ)$, or
		\item[\rm (2)] $H$ is dihedral and $G$ is contained in the normalizer of a Cartan subgroup of $\GL_2(\ZZ/p\ZZ)$, but not the Cartan itself, or
		\item[\rm (3)] $H$ is isomorphic to $A_4$, $S_4$ and $A_5$.
	\end{enumerate}
In case {\rm (2)}, $p$ must be odd. In case {\rm (3)}, $p$ must be relatively prime to 6, 6, and 30 respectively. 
\end{proposition}

We will say more about the Cartan subgroups when we discuss elliptic curves with complex multiplication in Section \ref{subsec:CMEllipticCurves}. 

Before moving on from Galois representations attached to elliptic curves, we draw attention to the fact that we can combine mod $p^k$ representations using inverse limits to define the $p$-adic Galois representations 

\[
\rho_{E,p^\infty}\colon \Gal(\overline\QQ/\QQ) \to \GL_2(\ZZ_p).
\]
% and
% \[
% \rho_{E}\colon \Gal(\overline\QQ/\QQ) \to \GL_2(\widehat\ZZ).
% \]
% \jeremy{I think we have a notation conflict. Sometimes we write $\bar\rho_{E,n}$ for the mod $n$ Galois representation, and sometimes we just write $\rho_{E,n}$. If $n = p$, it's not unclear if we're talking about the mod $p$ Galois representation or the $p$-adic Galois representation. I don't think we ever use this notation to refer to the adelic Galois representation, and the only places we refer to the $p$-adic Galois representation are in this section.}

An important point, for our purposes, is that if $p\geq 5$, then the group $\SL_2(\ZZ_p)$ has no proper closed subgroups whose image is $\SL_2(\ZZ/p\ZZ)$ under the standard reduction map. As Serre explains, this leads to the following proposition.

\begin{prop}\label{prop:SurjModp}{\rm \cite[Section IV]{SerreAbelReps}}
Let $E/\QQ$ be an elliptic curve and let $p\geq 5$ be a prime. If $\bar\rho_{E,p}$ is surjective, then $\rho_{E,p^{\infty}}$ is also surjective. 
\end{prop}

To see how this breaks down when $p=2$ or $3,$ the reader is encouraged to see \cite{Dokchitser2^n} and \cite{ElkiesMod3}.

Lastly, we note that given an elliptic curve over $\QQ,$ the group $\Im\bar\rho_{E,n}$ must have a few special properties. 

\begin{definition}
	Let $n\geq 2$ be a positive integer and let $G$ be a subgroup of $\GL_2(\ZZ/n\ZZ)$. 
	We say that $G$ is an {\bfseries admissible} group if 
	\begin{itemize}
		\item $\det(G) =(\ZZ/n\ZZ)^\times$, and 
		\item $G$ contains an element of determinant $-1$ and trace $0$ that fixes a point of order $n$ inside of $(\ZZ/n\ZZ)^2.$
	\end{itemize}
\end{definition}

\begin{proposition}{\rm \cite[Proposition 2.2]{Zyw}}
\label{prop:admissible}
	Let $n\geq 2$ be an integer and let $E/\QQ$ be an elliptic curve. 
	Then, $\Im\bar\rho_{E,n}$ is an admissible subgroup of $\GL_2(\ZZ/n\ZZ)$.
\end{proposition}

\subsubsection{Elliptic curves with complex multiplication}\label{subsec:CMEllipticCurves}

Elliptic curves come in two distinct types depending on their endomorphism rings. 
Given an elliptic curve $E$, defined over a field of characteristic 0, the endomorphism ring of $E$ over $\overline{\QQ}$ is either isomorphic to $\ZZ$ or an order of an imaginary quadratic field, usually denoted by $\mathcal{O}$. 
When the endomorphism ring is larger than $\ZZ$ we say that $E$ has {\it \bfseries complex multiplication} by $\mathcal{O}$. 
Throughout this section we follow the work done in \cite[Section 12]{RSZB}. 
Many of the results we use were first proven in \cite{ALR-CMImages}. A reader looking for an introduction to elliptic curves with complex multiplication should see \cite[Chapter II]{Silv2}.

One way to think about elliptic curves with complex multiplication is as elliptic curves with additional symmetries. 
These added symmetries manifest themselves in many ways. 
They endow elliptic curves with complex multiplication with many interesting properties that make them unique among elliptic curves in general. 
Of particular interest to us is that the Galois representations attached to elliptic curves with complex multiplication behave very differently than those without complex multiplication. 
%For example, if $E/\QQ$ does not have complex multiplication, then $\Im\rho_E$ has finite index inside of $\GL_2(\widehat{\ZZ})$, while if $E/\QQ$ does have complex multiplication, the index is infinite.

We introduce some notation to state results about the Galois representations attached to elliptic curves with complex multiplication.

Given $\mathcal{O}$, an order of a quadratic imaginary field $K$. We define the {\it \bfseries adelic Cartan subgroup associated to $\mathcal{O}$} to be
\[
\mathcal{C}_\mathcal{O} = \lim_{\leftarrow} (\mathcal{O}/N\mathcal{O})^\times
\]
where $N$ is a positive integer and the inverse limit is taken with respect to divisibility. 

Next, we let $\mathcal{O}_K$ be the maximal order inside of $K$ and we let $f = [\mathcal{O}_K:\mathcal{O}]$ be the conductor of $\mathcal{O}$. 
Continuing, we let $D = {\rm disc}(\mathcal{O}) = f^2{\rm\, disc}\mathcal{O}$, and \[
\phi = \begin{cases}
	f &\hbox{$D$ is odd,}\\
	0 & \hbox{otherwise}.
\end{cases}
\]
Then, we let 
\[
\omega = \frac{\phi+\sqrt{D}}{2} \hbox{ and }\delta = \frac{D-\phi^2}{4}
\]
so that $\mathcal{O} = {\rm Span}_\ZZ\{1,\omega\}$ with $\omega^2-\phi\omega-\delta = 0$. 
We can now define the level $N$ Cartan subgroup associated to $\mathcal{O}$ as 
\[
\mathcal{C}_\mathcal{O}(N) = \left\{\begin{pmatrix}a+b\phi&b\\ \delta b & a \end{pmatrix} \colon a,b\in\ZZ/N\ZZ\hbox{ and } a^2+ab\phi-\delta b^2\in(\ZZ/N\ZZ)^\times \right\}\subseteq\GL_2(\ZZ/N\ZZ).
\]
Taking inverse limits as $N$ runs over the positive integers ordered by divisibility of the level $N$, we can define 
\[
\mathcal{C}_\mathcal{O}(\widehat\ZZ) = \lim_{\leftarrow} \mathcal{C}_\mathcal{O}(N)\subseteq\GL_2(\widehat\ZZ).
\]
The group $\mathcal{C}_\mathcal{O}(\widehat\ZZ)$ is a closed subgroup of $\GL_2(\widehat\ZZ)$ that is isomorphic to $\mathcal{C}_\mathcal{O}$ under the isomorphism
\[
a+b\omega\mapsto \begin{pmatrix}a+b\phi&b\\ \delta b & a \end{pmatrix}.
\]

Next, we define 
\[
\mathcal{N}_\mathcal{O}(N) = \left\langle \mathcal{C}_\mathcal{O}(N), \begin{pmatrix}-1&0\\\phi&1 \end{pmatrix} \right\rangle
\]
and let $\mathcal{N}_\mathcal{O}(\widehat\ZZ)\subseteq \GL_2(\widehat\ZZ)$ and 
$\mathcal{N}_\mathcal{O}(\ZZ_p)\subseteq \GL_2(\widehat\ZZ_p)$
be the usual inverse limits of the $\mathcal{N}_\mathcal{O}(N)$. 

\begin{remark}
Frequently the group $\mathcal{N}_\mathcal{O}$ is called the normalizer of $\mathcal{C}_\mathcal{O}.$ 
It turns out that the group $\mathcal{N}_\mathcal{O}(\ZZ_p)$ is the normalizer of $\mathcal{C}_\mathcal{O}(\ZZ_p)$ in $\GL_2(\ZZ_p)$ (by \cite[Proposition 5.6(2)]{ALR-CMImages}), but $\mathcal{N}_\mathcal{O}$ is not the normalizer of $\mathcal{C}_\mathcal{O}$ in $\GL_2(\widehat{\ZZ}).$ See \cite[Remark 12.1.2]{RSZB}.
\end{remark}
%\jeremy{I don't think we ever refer to composite or adelic images of Galois for CM elliptic curves, and for that reason I don't really think we need all the adelic language, inverse limits, or the remark above. Is that right?}

Before moving on we make a few more observations about these groups that will be useful later on. 
First we note that by construction each of the groups $\mathcal{C}_\mathcal{O}(N)$ is an abelian group since $\mathcal{O}/N\mathcal{O}$ is abelian. 
In contrast, the groups $\mathcal{N}_{\mathcal{O}}$ are not abelian unless $p=2$. 
In order to compute the center of $\mathcal{N}_{\mathcal{O}}(N)$, it would be a simple matter of determining which matrices in $\mathcal{C}_{\mathcal{O}}(N)$ commute with $M = \left(\begin{smallmatrix}
	-1&0\\\phi&1
	\end{smallmatrix}\right).$
Let \[
A = \begin{pmatrix}
	a+b\phi&b\\\delta b& a
\end{pmatrix}\in \mathcal{C}_{\mathcal{O}}(N).
\]
Computing the entry in the first row, second column of $MA$ and $AM$ we see that $A$ commutes with $M$ if and only if $b=-b$. 
Thus, if $p\neq 2$, then $A\in Z(\mathcal{N}_{\mathcal{O}}(N))$ if and only if $A = \begin{pmatrix}
	a&0\\0&a
\end{pmatrix} = aI\in Z(\GL_2(\ZZ/N\ZZ)).$
From this we get the following lemma.

\begin{lemma}\label{lem:CenterN_O}
	Let $N > 2$ be an integer. Let $G$ be a subgroup of $\mathcal{N}_{\mathcal{O}}(N)$ such that \[
	\begin{pmatrix}
		-1&0\\\phi&1
	\end{pmatrix}\in G.
	\]
	Then the center of $G$, $Z(G)$, is exactly the set of scalar matrices in $G$. In other words,
	\[Z(G) = Z(\GL_2(\ZZ/N\ZZ)) \cap G. \]
\end{lemma}

% The next lemma we need is about the sizes of these groups in $p$-adic towers. 

% \begin{lemma}\label{lem:Size_p^k}
% 	Let $\mathcal{O}$ be an order of a quadratic imaginary field $K$, let $p$ be a prime and let $k\geq 2$ be a positive integer. 
% 	Then \[
% 		|\mathcal{C}_\mathcal{O}(p^k)|= p^{2(k-1)}|\mathcal{C}_\mathcal{O}(p)|,\hbox{and consequently }
% 		|\mathcal{N}_\mathcal{O}(p^k)|= p^{2(k-1)}|\mathcal{N}_\mathcal{O}(p)|.
% 	\]
% \end{lemma}

% \begin{proof}
% Consider the map $\pi\colon \mathcal{C}_\mathcal{O}(p^k)\to \mathcal{C}_\mathcal{O}(p)$ given by component-wise reduction. 
% The kernel of this map is exactly the set of 
% \[
% \begin{pmatrix}
% 	a+b\phi&b\\\delta b& a
% \end{pmatrix}\in \mathcal{C}_{\mathcal{O}}(N)
% \]
% congruent to the identity $\bmod\;p$. 
% Looking at the second column, this occurs exactly when $a\equiv 1 \bmod p$ and $b\equiv 0 \bmod p$. 
% So, the kernel of this map has size $(p^{k-1})^2$ and the result follows from the first isomorphism theorem. 
% \end{proof}

The last few theorems we need will help us pin down the exact image of the Galois representations attached to elliptic curves with complex multiplication. 
We will state the following theorems for elliptic curves with complex multiplication defined over $\QQ$, but we note that both \cite[Section 12]{RSZB} and \cite{ALR-CMImages} handle the case where $E$ is defined over a number field. 

\begin{proposition}\label{prop:CM_preim}{\rm \cite[Proposition 12.1.4]{RSZB}}
Let $E/\QQ$ be an elliptic curve with complex multiplication by $\mathcal{O}$, let $p$ be a prime, and let
\[e = \begin{cases}
4& \hbox{ if } p=2 \\
3 &\hbox{ if } p=3 \\
1 &\hbox{ otherwise.}
\end{cases}
\]
Then $\Im\rho_{E,p^\infty}$ is the inverse image of $\Im\bar\rho_{E,p^e}$ under the reduction map $\mathcal{N}_\mathcal{O}(\ZZ_p)\to\mathcal{N}_\mathcal{O}(\ZZ/p^e\ZZ)$.
\end{proposition}

Proposition \ref{prop:CM_preim} will be exactly what we need in order to understand how the division fields change as we go up the $p$-adic tower. This will be useful in Section \ref{sec:PrimePowerLevel}.

\begin{proposition}\label{prop:CMMaximalImage}{\rm \cite[Theorem 1.2(4)]{ALR-CMImages}}
Let $E/\QQ$ be an elliptic curve with complex multiplication by $\mathcal{O}\neq \ZZ[\zeta_3].$ 
{If $p$ does not divide $2 \disc(\mathcal{O})$}, then there is a choice of basis such that $\Im\rho_{E,p^\infty} = \mathcal{N}_\mathcal{O}(\ZZ_p)$.
\end{proposition}

\subsection{Modular Curves}\label{subsec:ModularCurves}

Modular curves are the main objects that we will use to classify the elliptic curves over $\QQ$ with a given admissible group as the image of their mod $n$ representation. 
Given a natural number $n\geq 2$ and an admissible subgroup $G\subseteq \GL_2(\ZZ/n\ZZ)$ there is a smooth, projective, and geometrically irreducible curve defined over $\QQ$ denoted $X_G$ whose $\QQ$-rational points classify elliptic curves with the property that $\Im\bar\rho_{E,n}$ is conjugate to a {\it subgroup} of $G.$ 
Here we emphasize that the image of $\Im\bar\rho_{E,n}$ need not be all of $G$ in order to have a corresponding point on $X_G.$ Indeed, since subgroups of nilpotent groups are nilpotent, this will allow us to focus on finding the maximal admissible nilpotent groups of a given level.

The nature of this correspondence depends on whether $-I\in G$ or not,
but if $G$ were a nilpotent group that did not contain $-I$, then adding $-I$ would preserve nilpotency.
For this reason, we can assume that $-I\in G$. Then the curve $X_G$ always comes with a natural morphism
\[
\pi_G\colon X_G\to \mathbb{P}^1_\QQ
\]
such that an elliptic curve $E/\QQ$ with $j$-invariant $j_E\not\in\{0,1728\}$, has $\Im\bar\rho_{E,n}$ conjugate to a subgroup of $G$ if and only if $j_E = \pi_G(P)$ for some $P\in X_G(\QQ)$.
The interested reader should see \cite[Subsection 2.3]{RSZB} to see more details when $-I\not\in G$. %\jeremy{(I changed the reference to RSZB since Section 2 of RZB is slightly wrong.)}

We end this section by emphasizing that a complete classification of the points on these curves would give a corresponding classification of $\im\bar\rho_{E,n}$ for every elliptic curve $E/\QQ$.
That is, if we can determine all of the maximal nilpotent subgroups $H$ of $\GL_2(\ZZ/n\ZZ)$ and classify all the rational points on the corresponding $X_H$'s, then we would have classified all the elliptic curves with nilpotent $n$-division fields.

\section{Near coincidences}\label{sec:near}
% Before starting the proof in earnest, we will start by taking care of the case when $E$ has complex multiplication. The base case will break down into a number of subcases. We will have to handle the cases when $p=2$ and $3$ separately, but even further in the case when $p\geq 5,$ we will have to break this into cases depending on $\Im
% \bar\rho_{E,p}$ according to Proposition \ref{prop:ProjIm}. If $E/\QQ$ is an elliptic curve with complex multiplication and $p\geq5$ then $\QQ(E[p^2]) = \QQ(E[p],\zeta_{p^2})$ cannot hold since by Lemma \ref{lem:Size_p^k} and Proposition \ref{prop:CM_preim} the extension $\QQ(E[p^2])/\QQ(E[p])$ must be degree $p^2$. If $E/\QQ$ is an elliptic curve with complex multiplication and $p=2$ or 3, then we can use the fact that all of the possible images of $\rho_{E,p^{\infty}}$ are listed in \cite[Tables 18--22]{RSZB}. With this information, we can completely answer the question of near coincidences of division fields for elliptic curves with complex multiplication.

In order to classify near coincidences, it will be useful to be able to detect them using the image of the corresponding Galois representation. With this in mind, we give the following definition.

\begin{definition}
Let $G\subseteq\GL_2(\ZZ/n\ZZ)$ with surjective determinant and suppose that $m \mid n$. We say that $G$ {\bfseries represents a near coincidence} of level $(n,m)$ if 
\[
(G\cap\SL_2(\ZZ/n\ZZ)) \cap \Ker(\pi) = \{I\},
\]
where $\pi\colon\GL_2(\ZZ/n\ZZ)\to\GL_2(\ZZ/m\ZZ)$ is the standard componentwise reduction map. 
\end{definition}

\begin{remark}
The idea behind this definition is that classically we know that $(G\cap\SL_2(\ZZ/n\ZZ))$ fixes $\QQ(\zeta_n)\subseteq\QQ(E[n])$ and $\Ker(\pi)$ fixes $\QQ(E[m])\subseteq\QQ(E[n]).$ Therefore, $(G\cap\SL_2(\ZZ/n\ZZ)) \cap \Ker(\pi)$ should fix $\QQ(E[m],\zeta_n)$. Thus, the only way that $\QQ(E[n]) = \QQ(E[m],\zeta_n)$ is if $(G\cap\SL_2(\ZZ/n\ZZ)) \cap \Ker(\pi) = \{I\}.$

Of course this definition requires that $m\mid n,$ but that lines up with the original definition of near coincidence.
\end{remark}

The proof of Theorem \ref{thm:Main_NearCoin} will be done by first considering the case of prime level and then moving on to the case of prime power level. We will have to break the prime level case down into 3 smaller cases. These cases consist of when $p=2,$ $p=3$, or $p\geq5.$ 
%\jeremy{The use of the variable $n$ in Theorem 1.3 does not mesh nicely with the way we have used $n$ in the preceding paragraph. I suggest rewording to refer to ``prime level'' and ``prime power level''.}

\subsection{Proof of Theorem \ref{thm:Main_NearCoin} for prime levels}\label{subsec:n=1}

We will start the case when $n=1$ of the theorem by dealing with primes $p \geq 5$
and break the argument into cases depending on $\Im
\bar\rho_{E,p}$ based on Proposition \ref{prop:ProjIm}. First we handle the case that $\bar\rho_{E,p}$ is surjective.

\begin{prop}
\label{prop:Main_NearCoin_surj} Assume that $E/\Q$ is an elliptic curve, $p \geq 5$ is prime, and $\Im\bar\rho_{E,p}= \GL_2(\ZZ/p\ZZ)$. Then $\QQ(E[p^{2}]) \ne \QQ(E[p],\zeta_{p^{2}})$.
\end{prop}
\begin{proof}
By Proposition \ref{prop:SurjModp}, we have that $\Im \rho_{E,p^{\infty}} = \GL_{2}(\ZZ_{p})$ and hence $\Im\bar\rho_{E,p^{2}} = \GL_{2}(\ZZ/p^{2} \ZZ)$. Hence
\[
  [\QQ(E[p^{2}]) : \QQ] = |\GL_{2}(\ZZ/p^{2} \ZZ)| = p^{5} (p-1)^{2} (p+1),
\]
while
\begin{align*}
  [\QQ(E[p],\zeta_{p^{2}})] &\leq [\QQ(E[p],\zeta_{p^{2}}) : \QQ(E[p])]
  [\QQ(E[p]) : \Q] \leq p(p-1) \cdot |\GL_{2}(\ZZ/p \ZZ)|\\ 
  &= p^{2} (p-1)^{3} (p+1) < p^{5} (p-1)^{2} (p+1).
\end{align*}
This proves the claim.
\end{proof}

Next we handle the case that $\Im\bar\rho_{E,p}$ is a proper subgroup whose order is a multiple of $p$.

\begin{prop}
\label{prop:Main_NearCoin_Borel}
Assume that $E/\Q$ is an elliptic curve, $p \geq 5$ is prime, and $\Im\bar\rho_{E,p}$ is a proper subgroup of $\GL_{2}(\ZZ/p\ZZ)$ whose order is a multiple of $p$. Then $\QQ(E[p^{2}]) \ne \QQ(E[p],\zeta_{p^{2}})$.
\end{prop}
\begin{proof}
In this case, the image of $\Im\bar\rho_{E,p}$ is contained in a Borel subgroup by Proposition~\ref{prop:ProjIm} and hence $E$ has a cyclic $p$-isogeny defined over $\Q$.
By the classification of Mazur, we have that
$p \in \{ 5, 7, 11, 17, 19, 43, 67, 163 \}$. Let $H = \im \rho_{E,p^{\infty}}$.

If $E$ has complex multiplication, choose a basis so that $H \subseteq \mathcal{N}_{\mathcal{O}}(\ZZ_{p})$. By Proposition~\ref{prop:CM_preim}, $H$ is the full preimage in $\mathcal{N}_{\mathcal{O}}(\ZZ_{p})$ of $\Im\bar\rho_{E,p}$.
Since $\mathcal{N}_{\mathcal{O}}(\ZZ_{p})$ contains all matrices of the form $\begin{bmatrix} 1+kp & 0 \\ 0 & 1-kp \end{bmatrix}$ with $k \in \ZZ/p\ZZ$ and this implies that $\QQ(E[p^{2}])/\QQ(E[p],\zeta_{p^{2}})$ is a non-trivial extension. 

If $E$ does not have complex multiplication, then Theorem 1.1.6 of \cite{RSZB} implies that $X_{H}$ is isomorphic to $\PP^{1}$ or a positive rank elliptic curve,
or that $X_{H}$ is listed in Table 1 of \cite{RSZB} as the other cases are ruled out by the assumption that $p \mid |\Im\bar\rho_{E,p}|$. If we assume that
$\QQ(E[p^{2}]) = \QQ(E[p],\zeta_{p^{2}})$, then $p^{2} \nmid |\im\bar\rho_{E,p^{2}}|$
and this implies that the index of $H$ in $\GL_{2}(\ZZ_{p})$ is a multiple of $p^{3}$.
Since $p \geq 5$, this implies that the index is $\geq 125$ and this implies that $X_{H}$ has genus $\geq 2$. There are groups $H$ contained in a Borel subgroup of $\GL_{2}(\ZZ/p\ZZ)$ listed in Table 1 of \cite{RSZB} for $p = 11$, $17$ and $37$, but in no case is the index of $H$ a multiple of $p^{3}$. This concludes the proof.
\end{proof}

Next we handle the case that $\Im\bar\rho_{E,p}$ is contained in the normalizer of a split Cartan subgroup. Here we divide the argument into the CM case and the non-CM case.

\begin{prop}
\label{prop:Main_NearCoin_splitcartan}
Let $E/\QQ$ be an elliptic curve and $p\geq 5$ a prime such that $\Im\bar\rho_{E,p}$ is contained in the normalizer of a split Cartan subgroup of $\GL_2(\ZZ/p\ZZ)$. Then $\QQ(E[p^2]) \ne \QQ(E[p],\zeta_{p^2})$.
\end{prop}
\begin{proof}
We first will show that $E$ has complex multiplication. If $E$ does not have complex multiplication, then the results of \cite{BiluParentRebolledo} rule out $p = 11$
and $p \geq 17$, and the results of \cite{Balakrishnan} rule out $p = 13$.
Hence $p \in \{ 5, 7 \}$. If $H = \Im\bar\rho_{E,p^{2}}$ and $\QQ(E[p^{2}]) = \QQ(E[p],\zeta_{p^{2}})$, then $H$ is an index $p^{3}$ subgroup of the full preimage in $\GL_{2}(\ZZ/p\ZZ)$ of $\Im\bar\rho_{E,p}$ and hence $|H|$ divides
$\frac{|\GL_{2}(\ZZ/p^{2} \ZZ)|}{p^{4} (p+1)/2}$, since $\frac{p(p+1)}{2}$ is the index
of the normalizer of the split Cartan in $\GL_{2}(\ZZ/p\ZZ)$. In addition,
by Proposition~\ref{prop:admissible}, $H$ must also be an admissible group. Finally, $H$ must be non-abelian, since by \cite{AbelDivFields} $\QQ(E[n])/\QQ$ is non-abelian if $n > 8$.

For these two primes $p$, we use Magma to search for admissible subgroups
$H$ of $\GL_{2}(\ZZ/p^{2} \ZZ)$ so that (i) the order of $H$ divides
$\frac{|\GL_{2}(\ZZ/p^{2} \ZZ)|}{p^{4} (p+1)/2}$, and (ii) $H$ is non-abelian. 
For $p = 5$ there are two conjugacy classes of such subgroups and for $p = 7$ there are four. In both cases, all such subgroups are contained in the normalizer of a split Cartan subgroup of $\GL_{2}(\ZZ/p^{2} \ZZ)$. (The code for these calculations can be found in the file {\tt prop35.m} at \cite{code}.) The rational points on the modular curve corresponding to the normalizer of
the split Cartan subgroup mod $25$ are determined by Momose and Shimura in
\cite{MomoseShimura}, and the mod $49$ case is handled by Momose in \cite{Momose}. In both cases, all such elliptic curves have CM. It follows that $E$ has complex multiplication.

Now we show that if $E$ has complex multiplication, then $\QQ(E[p^{2}]) \ne \QQ(E[p],\zeta_{p^{2}})$. The here is identical to that in the proof of Proposition~\ref{prop:Main_NearCoin_Borel}. Choose a basis so that $\Im\bar\rho_{E,p^{2}} \subseteq \mathcal{N}_{\mathcal{O}}(\ZZ_{p})$. By Proposition~\ref{prop:CM_preim}, $\Im\bar\rho_{E,p^{2}}$
is the full preimage in $\mathcal{N}_{\mathcal{O}}(\ZZ_{p})$ of $\Im\bar\rho_{E,p}$.
Therefore, $\Im\bar\rho_{E,p^{2}}$ contains all matrices of the form $\begin{bmatrix} 1+kp & 0 \\ 0 & 1-kp \end{bmatrix}$ with $k \in \ZZ/p\ZZ$ and this implies that $\QQ(E[p^{2}])/\QQ(E[p],\zeta_{p^{2}})$ is a non-trivial extension. 
\end{proof}

The last case that has to be dealt with is the case when $\Im\bar\rho_{E,p}$ is contained inside of the normalizer of a non-split Cartan subgroup of $\GL_2(\ZZ/p\ZZ)$ for $p\geq 5$. One important fact in this case is that any such elliptic curve has to have potentially supersingular reduction at $p$ from either \cite[Appendix B]{LeFournLemos} or \cite[Proposition 1.13]{Zyw}. 

\begin{prop}\cite[p. 312]{Serre72}
Let $E/\QQ$ be an elliptic curve with potential good reduction at $p\geq 5$ and discriminant $\Delta$. Then, $E$ acquires good reduction over $\QQ(\sqrt[12]{\Delta})$ at all primes over $p$. 
\end{prop}

In \cite[p. 312]{Serre72}, Serre explains that not only does $E$ gain good reduction over $\QQ(\sqrt[12]{\Delta}),$ but also that ${\rm ord}_p(\Delta)\in \{ 0,2,3,4,6,8,9,10 \}$. Thus, if $\mathfrak{p}$ is a prime above $p$ in $\QQ(\sqrt[12]{\Delta})$, then $e(\mathfrak{p}/p)\in \{1,2,3,4,6\}.$

Even with this in hand, for $p=5$ we will have to rely on the classification of rational points provided in \cite{RSZB}. 

\begin{prop}
Let $E/\QQ$ with the property that $\im\bar\rho_{E,5}$ is contained in the normalizer of the non-split Cartan subgroup mod $5$. Then $\QQ(E[25]) \ne \QQ(E[5],\zeta_{25})$.
\end{prop}
\begin{proof}
We search $\GL_2(\ZZ/25\ZZ)$ for groups that represent $(25,5)$ near coincidences. We then take the ones that are maximal with respect to containment (up to conjugation) and check if they have points. These groups are exactly the group with RSZB labels 
\begin{center}
\texttt{25.625.36.1, 25.1250.76.1, 25.2500.156.3, 25.2500.156.2, 25.3750.236.2, 
25.3750.236.1}
\end{center}
Using the data in \cite{RSZB}, we see that there are no non-cuspidal $\QQ$-rational points on any of these curves and so there are no elliptic curves over $\QQ$ with a $(25,5)$ near coincidence. 
\end{proof}

We now turn to the case that $p \geq 7$. 

\begin{prop}\label{prop:TotallyRam}
Let $E/\QQ$ be an elliptic curve and let $p\geq 7$ be a prime such that $\Im\bar\rho_{E,p}$ is a subgroup of the normalizer of a non-split Cartan subgroup $\GL_2(\ZZ/p\ZZ).$ Further, let $\Delta$ be the discriminant of $E$ and suppose that $\QQ(E[p^2]) = \QQ(E[p],\zeta_{p^2})$. Let $\mathfrak{p}$ be a prime over $p$ in $\QQ(\sqrt[12]{\Delta}).$ The extension $\QQ(\sqrt[12]{\Delta},E[p^2])/\QQ(\sqrt[12]{\Delta})$ is a degree $2p(p^2-1)$ extension that is totally ramified at $\mathfrak{p}$.
\end{prop}

\begin{proof}
	Suppose towards a contradiction that there is a prime $\mathfrak{P}$ above $\mathfrak{p}$ such that $e(\mathfrak{P}/\mathfrak{p})$ is a proper divisor of $2p(p^2-1)$. Since $\Im\bar\rho_{E,p}$ is a subgroup of the normalizer of a non-split Cartan subgroup of $\GL_2(\ZZ/p\ZZ)$, we know that $[\QQ(E[p]):\QQ]$ is a divisor of $2(p^2-1)$. 
	Further, because we assumed that $\QQ(E[p^2]) = \QQ(E[p],\zeta_{p^2})$, we know that $[\QQ(E[p^2]):\QQ]$ is a divisor of $2p(p^2-1)$. Thus $\QQ(\sqrt[12]{\Delta},E[p^2])/\QQ(\sqrt[12]{\Delta})$ has degree dividing $2p(p^2-1)$.  
	Now, $e(\mathfrak{P}/\mathfrak{p})$ is a proper divisor of $| \bar\rho_{E,p^2}(\Gal(\overline{\QQ}/\QQ(\sqrt[12]{\Delta})) |$ which itself divides $2p(p^2-1).$ Thus \[
	e(\mathfrak{P}/\mathfrak{p})\leq p(p^2-1)\hbox{ and } e(\mathfrak{P}/p) = e(\mathfrak{P}/\mathfrak{p})e(\mathfrak{p}/p)\leq p(p^2-1)\cdot 6<p^2(p^2-1).
	\]
	But the main result (Theorem 1.1) of Hanson Smith's paper \cite{HansonSmith} says that $e(\mathfrak{P}/p)\geq p^4-p^2 = p^2(p^2-1)$ giving us our contradiction. %\jeremy{I modified the last sentence to highlight the result we're using a bit more.}
\end{proof}

We now explain how the previous proposition rules out the possibility of a $(p^{2},p)$ near coincidence.

\begin{prop}\label{prop:n=1}
Let $E/\QQ$ be an elliptic curve and let $p\geq 7$ be a prime such that $\Im\bar\rho_{E,p}$ is a subgroup of the normalizer of a non-split Cartan subgroup $\GL_2(\ZZ/p\ZZ).$ Then $\QQ(E[p^2]) \neq \QQ(E[p],\zeta_{p^2})$.
\end{prop}

\begin{proof}
Suppose towards a contradiction that $\QQ(E[p^2]) = \QQ(E[p],\zeta_{p^2})$. By the previous proposition, we know that $\QQ(\sqrt[12]{\Delta},E[p^2])/\QQ(\sqrt[12]{\Delta})$ is totally ramified at $\mathfrak{p}$. The assumption that 
$\QQ(E[p^{2}]) = \QQ(E[p],\zeta_{p^{2}})$ implies that $\Im\bar\rho_{E,p^{2}}$ has order dividing $2p(p^{2} - 1)$. On the other hand, Proposition~\ref{prop:TotallyRam} implies that
\[
  |\bar\rho_{E,p^{2}}(\Gal(\overline{\QQ}/\Q(\sqrt[12]{\Delta})))| = 2p(p^{2} - 1).
\]
It follows from this that
\[
  \Im\bar\rho_{E,p^{2}} = \bar\rho_{E,p^{2}}(\Gal(\overline{\QQ}/\Q(\sqrt[12]{\Delta})))
\]
and hence
\[
\Gal(\QQ(\sqrt[12]{\Delta},E[p^2])/\QQ(\sqrt[12]{\Delta}))\simeq \Im\bar\rho_{E,p^2}.\]
However, a Galois extension that is totally ramified at a prime over $p$ must have a Galois group which is an extension of a finite $p$-group (the wild inertia group) by a finite cyclic group of order coprime to $p$ (the tame inertia group). This means that for any prime $q\neq p$, the $q$-Sylow subgroup of $\Im\bar\rho_{E,p^2}$ must be cyclic. In particular, the Sylow $2$-subgroup of
$\Im\bar\rho_{E,p^{2}}$ must be cyclic. However, by \cite{AbelDivFields}, $\Q(E[p])/\Q$ must be nonabelian, and so the projectivization of $\Im\bar\rho_{E,p}$ must have a dihedral Sylow $2$-subgroup. This is a contradiction.
\end{proof}

The last case that remains here is the case that $\Im\bar\rho_{E,p}$ is contained in an exceptional group (i.e. the image in $\PGL_{2}(\FF_{p})$ falls into case 3 of Proposition~\ref{prop:ProjIm}).

\begin{prop}\label{prop:Exceptional}
Let $E/\QQ$ be an elliptic curve and let $p\geq 5$ be a prime such that $\Im\bar\rho_{E,p}$ is contained in an exceptional subgroup. Then $\QQ(E[p^2]) \ne \QQ(E[p],\zeta_{p^2})$. 
\end{prop}
\begin{proof}
Serre showed \cite[Section 8.4, Lemme 18]{SerreQuelques} that $\Im\bar\rho_{E,p}$ can only
be an exceptional subgroup if image in $\PGL_{2}(\FF_{p})$ is isomorphic to $S_{4}$,
$p \leq 13$ and $p \equiv 3 \text{ or } 5 \pmod{8}$. 

The elliptic curves with $\Im\bar\rho_{E,13}$ contained in an exceptional subgroup were determined in \cite{Balakrishnan2} and in \cite{RSZB} it was shown that for each such elliptic curve, $\Im\bar\rho_{E,13^{2}}$ contains all matrices $\equiv I \pmod{13}$, which implies that $\QQ(E[p^2]) \ne \QQ(E[p],\zeta_{p^2})$. 

For $p = 11$,
the elliptic curves with $\Im\bar\rho_{E,11}$ contained an exceptional subgroup were determined by Ligozat \cite[Proposition II.4.4.8.1]{Ligozat} and the only possibility is $j(E) = 0$. For such an elliptic curve,
we cannot have $\QQ(E[11^{2}]) = \QQ(E[11],\zeta_{11^{2}})$ by Proposition~\ref{prop:CM_preim}.

For $p = 5$, if $\Im\bar\rho_{E,5}$ is contained in an exceptional subgroup, then it either equals an exceptional subgroup, or is contained in a Borel or the normalizer of a split Cartan. The latter two cases are impossible by Proposition~\ref{prop:Main_NearCoin_Borel} and Proposition~\ref{prop:Main_NearCoin_splitcartan}.
If the image is the exceptional subgroup mod $5$, then a group theory computation with Magma shows that the mod $25$ image of Galois has RSZB label \texttt{25.625.36.1}, which is shown
to be impossible in \cite[Subsection 8.6]{RSZB}.
\end{proof}

In the case when $p=2$ and $3$, we search $\GL_{2}(\ZZ/p^2\ZZ)$ for subgroups that represent a near coincidence of level $(p^2,p)$ and then compute the maximal groups ordered by containment up to conjugation. Doing this yields curves with labels 
\begin{center}
\texttt{4.16.0.1}, \texttt{4.16.0.2}, and \texttt{4.48.0.3},
\end{center}
when $p=2$, and 
\begin{center}
\texttt{9.27.0.1}, \texttt{9.162.4.1}, and \texttt{9.324.10.1}
\end{center}
when $p=3.$
Thanks to \cite{lmfdb, RSZB, RZB} we know that \texttt{4.16.0.2}, \texttt{9.162.4.1}, and \texttt{9.324.10.1} do not have any rational points and so can be omitted.
To summarize, we have the following proposition:
\begin{prop}\label{prop:nearsummary}
Let $E/\QQ$ be an elliptic curve and let $p\in\ZZ$ be a prime such that $\QQ(E[p^2]) = \QQ(E[p],\zeta_{p^2}).$ Then, $p=2$ or $p=3$ and $E$ corresponds to a rational point on one of the modular curves with RSZB labels
\texttt{4.48.0.3}, \texttt{4.16.0.2}, or \texttt{9.27.0.1}.
%4.48.0.3: Q(E[2]) = Q and Q(E[4]) = Q(i), and 
%4.16.0.2: Q(E[2]) = Q(E[4]);
%9.27.0.1: Q(E[9]) = Q(E[3],zeta_9)
\end{prop}

\subsection{Proof of Theorem \ref{thm:Main_NearCoin} for prime power levels} 

To start this section, we push a little further in the cases where $p=2$ and $3.$ These searches were carried out using a Magma script, \texttt{NearCoin.m}, which can be found in \cite{code}. In both of these cases, we can have near coincidence between the $p^2$- and $p$-division fields, but what about the $p^3$- and $p^2$-division fields? 

So we search for groups that represent $(8,4)$ and $(27,9)$ coincidences. In the first case, we find that the maximal groups that represent an $(8,4)$ coincidence all have genus 1 or higher. Using the data in \cite{RZB}, we know that this means that there is no elliptic curve without complex multiplication that has these images, and we have already completely dealt with the CM case. 

When considering $(27,9)$ coincidences, the maximal groups are the ones with RSZB labels
\begin{center}\texttt{
27.729.43.1, 27.4374.280.4, 27.4374.280.1, 27.4374.280.3, 27.4374.280.2, 
27.8748.568.2, 27.8748.568.5, 27.8748.568.1, 27.8748.568.3.}
\end{center}
Again, \cite{RSZB} says that the corresponding modular curves have no non-cuspidal $\QQ$-rational points and so there are no elliptic curves over $\QQ$ with a $(27,9)$ near coincidence.

\begin{remark}
In \cite{RSZB}, it was shown that \texttt{27.729.43.1} cannot occur as the image of $\rho_{E,27}$ for any elliptic curve over $\QQ$ by writing down the canonical model of this modular curve in $\mathbb{P}^{42}$ and showing it has no mod $9$ points. The argument given above in Proposition~\ref{prop:TotallyRam} and Proposition~\ref{prop:n=1} can be modified to give a simpler proof that this modular curve has no rational points. In particular, one can show that if $E/\QQ$ has mod $9$ image contained in \texttt{9.27.0.1} (a supergroup of \texttt{27.729.43.1}), then ${\rm ord}_{3}(j(E)) \geq 7$. Since any elliptic curve with $j(E) \equiv 0 \pmod{3}$ has potentially supersingular reduction at $3$, the argument (using Theorem 1.1 of \cite{HansonSmith}) can proceed along similar lines.
\end{remark}

Next we prove the case of Theorem \ref{thm:Main_NearCoin} when $n\geq 2$ by induction.

\begin{prop}\label{prop:Near_Inductive_p_odd}
Suppose that $E/\QQ$ is an elliptic curve and $p>2$ is a prime such that $p^k$ divides $[\QQ(E[p^{n+1}]):\QQ(E[p^n])]$ for some $k\in\{1,2,3,4\}$ and $n\geq 1.$ Then, $p^k$ divides $[\QQ(E[p^{n+2}]):\QQ(E[p^{n+1}])]$
\end{prop}

\begin{proof}

Assume that $n\geq 1$ and $p^k$ divides $[\QQ(E[p^{n+1}]) : \QQ(E[p^n])]$ for some $k\in\{1,2,3,4\}$ This means that the set 
\[
S =\{A\in\Im\bar\rho_{E,p^{n+1}} \colon A\equiv I\bmod{p^n}\}
\]
must have size at least $p^k.$ Next, we let 
\[
\tilde{S} =\{A\in\Im\bar\rho_{E,p^{n+2}} \colon A\equiv I\bmod{p^{n+1}}\}.
\]
Our goal now is to show that there is an injective homomorphism from $S$ to $\tilde{S}.$ Doing this would allow us to conclude that $|S|$ divides $|\tilde{S}|$. 

We can represent an element of $S$ in the form $I + p^{n} X$ with $X$ in the additive group
$M_{2}(\Z/p\Z)$ of $2 \times 2$ matrices. Define $\phi : S \to \tilde{S}$ by $\phi(I + p^{n} X) = I + p^{n+1} X$. It is straightforward to see that this formula defines an injective homomorphism,
but it is not immediately clear that if $I + p^{n} X \in S$, then $I + p^{n+1} X \in \tilde{S}$.
We now justify this. 

If $I+p^nX\in S$ for some $X\in M_2(\ZZ/p\ZZ)$, then there is a $\sigma_0\in\Gal(\QQ(E[p^{n+1}]/\QQ))$ such that $\bar\rho_{E,p^{n+1}}(\sigma_0) = I+p^nX$. Further, there must be a $\sigma\in\Gal(\QQ(E[p^{n+2}])/\QQ)$ such that $\sigma\big|_{\QQ(E[p^{n+1}])} = \sigma_0$. In this case, since $\bar\rho_{E,p^{n+1}}(\sigma_0) \equiv I + p^nX$ we have that 
\[
\bar\rho_{E,p^{n+2}}(\sigma) = I+p^n\tilde{X}
\] 
for some $\tilde{X}\in M_2(\ZZ/p^2\ZZ)$ such that $\tilde{X} \equiv X \bmod p$. Then it must be that 
\begin{align}
\notag \bar\rho_{E,p^{n+2}}(\sigma^p) & \equiv (I+p^n\tilde{X})^p \bmod p^{n+2}\\
&\equiv I+p\cdot p^n\tilde{X}+\frac{1}{2}p(p-1)p^{2n}\tilde{X}^2+\cdots \bmod p^{n+2}\label{eq:p_neq_2}\\
\notag &\equiv I + p^{n+1} \tilde{X} \bmod p^{n+2}\\
\notag &\equiv I + p^{n+1} X \bmod p^{n+2}.
\end{align}
Thus $I+p^{n+1}X\in \tilde{S}$.

Since $|S|$ divides $|\tilde{S}|$, this forces $\QQ(E[p^{n+2}])/\QQ(E[p^{n+1}])$ to have degree at least $|S|$ and so
\[
p^k\mid[\QQ(E[p^{n+2}]):\QQ(E[p^{n+1}])].
\]
\end{proof}

\begin{remark}
Notice that if $p\geq 3$, then the statement that $p^2\mid [\QQ(E[p^{n+1}]):\QQ(E[p^n])]$ is equivalent to the fact that $\QQ(E[p^{n+1}]) \neq \QQ(E[p^n],\zeta_{p^{n+1}}).$ This is because the field extension $\QQ(E[p^{n+1}])/\QQ(E[p^n])$ is a Galois extension whose Galois group is isomorphic to a subgroup of the additive group $M_2(\ZZ/p\ZZ)$ of $2\times2$ matrices with entries in $\ZZ/p\ZZ$. The group $M_2(\ZZ/p\ZZ)$ has order $p^4$ and so a priori $[\QQ(E[p^{n+1}]):\QQ(E[p^n])]=p^k$ for some $k\in\{0,1,2,3,4\}.$ We omit the case when $k=0$ in Proposition \ref{prop:Near_Inductive_p_odd} since it is uninteresting. 

Next, we notice that 
\[
[\QQ(E[p^{n+1}]) : \QQ(E[p^n])] = [\QQ(E[p^{n+1}]) : \QQ(E[p^n],\zeta_{p^{n+1}}) ] [\QQ(E[p^n],\zeta_{p^{n+1}}) : \QQ(E[p^n])].
\]
But, by Theorem \ref{thm:MainCoin}, in this case $\zeta_{p^{n+1}}\not\in \QQ(E[p^n])$ and so $[\QQ(E[p^n],\zeta_{p^{n+1}}) : \QQ(E[p^n])] =p$.
Bringing it all together we see that
\begin{align*}
p^{2} \mid [\Q(E[p^{n}) : \Q(E[p^{n-1}])] &\iff p \mid [\Q(E[p^{n}]) : \Q(E[p^{n-1}], \zeta_{p^{n}})]\\
&\iff [\Q(E[p^{n}]) : \Q(E[p^{n-1}], \zeta_{p^{n}})] \ne 1.
\end{align*}
\end{remark}

Examining the proof of Proposition~\ref{prop:Near_Inductive_p_odd}, the term
$\frac{1}{2} p(p-1) p^{2n} \tilde{X}^{2}$ is $\equiv 0 \pmod{p^{n+2}}$ if $(p,n) \ne (2,1)$ but could fail if $p = 2$ and $n = 1$. With this in mind, we immediately get the following corollary.

\begin{corollary}\label{cor:Near_Inductive_p_even}
Suppose that $p=2$ and $E/\QQ$ is an elliptic curve such that $p^k$ divides $[\QQ(E[p^{n+1}]):\QQ(E[p^n])]$ for some $k\in\{1,2,3,4\}$ and $n\geq 2.$ Then, $p^k$ divides $[\QQ(E[p^{n+2}]):\QQ(E[p^{n+1}])]$
\end{corollary}

Thus, we find ourselves at the end. The work of \ref{subsec:n=1} together with Proposition \ref{prop:Near_Inductive_p_odd} and Corollary \ref{cor:Near_Inductive_p_even} completes the proof of Theorem \ref{thm:Main_NearCoin}.

\section{Nilpotent division fields of prime level}\label{sec:PrimeLevel}\

We are now ready to start classifying when the division fields of elliptic curves can give us nilpotent extensions of $\QQ.$ Before starting the classification in earnest, we will quickly remind the reader of some basic facts about nilpotent groups.

\subsection{Nilpotent Groups}
This subsection will only cover the very basics of subgroups series and nilpotent groups. 
Readers interested in more context should see \cite{ConradSubSeriesI,ConradSubSeriesII,DF,Isaacs}. 

\begin{definition}
	Let $G$ be a group. An ascending series 
	\[
	\{e\} = G_0 \subseteq G_1\subseteq G_2\subseteq\dots\subseteq G
	\]
	is called a \it{\bfseries central series} if for all $i$, $G_i\triangleleft G$ and $G_{i+1}/G_i\subseteq Z(G/G_i).$ Here $Z(G)$ is the center of $G$. 
	A descending series
	\[
	G = G_0\supseteq G_2 \supseteq G_2 \supseteq \dots \supseteq \{e\}
	\]
	is called a {\it \bfseries central series} if $G_i\triangleleft G$ and $G_i/G_{i+1}\subseteq Z(G/G_{i+1}).$
\end{definition}

\begin{definition}
	A group $G$ is called {\bfseries nilpotent} if it has a central series.
\end{definition}

\begin{theorem}\label{thm:Nilp}{\rm \cite[Theorem 1.26]{Isaacs}}
Let $G$ be a finite non-trivial group. The following are equivalent:
\begin{enumerate}
	\item $G$ is a nilpotent group.
	\item Every Sylow subgroup of $G$ is normal. 
	\item $G$ is the direct product of its Sylow subgroups.
	\item If $d$ divides $|G|$, then $G$ has a normal subgroup of order $d$.
\end{enumerate}
\end{theorem}

An immediate consequence of this result is that every abelian group and every finite $p$-group is nilpotent.

\begin{prop}\label{prop:NilpIsClosed}{\rm \cite[Theorem 5.7]{ConradSubSeriesI}}
Nilpotency is closed under subgroups, quotients, and direct products. 
\end{prop}

In general, given a group $G$ and a nilpotent normal subgroup $N$, it is not true that $G/N$ nilpotent implies that $G$ is nilpotent. However if $N \leq Z(G)$, this follows from Theorem 5.13 of \cite{ConradSubSeriesI}.

\begin{prop}\label{prop:TrivialCenter}
If $G$ is a finite non-trivial group such that $Z(G) = \{e\}$, then $G$ is not nilpotent.
\end{prop}

\begin{proof}
Suppose $G$ is a finite nontrivial group that is nilpotent. Then, by Theorem \ref{thm:Nilp} part 3, $G$ is the direct products of its $p$-Sylow subgroups. A classical result in group theory is that $p$-group have nontrivial centers and so $G$ must have a nontrivial center.
\end{proof}

\begin{example}\label{ex:WhenDnIsNilp}
	Let $D_n$ be the dihedral group of order $2n$. More specifically, let \[
	 D_n = \langle r,s \mid r^n = s^2 =e, srs^{-1} = r^{-1}\rangle.
	 \]
	 A classical result is that $D_n$ is nilpotent exactly when $n = 2^k$ for some $k\geq 2$. In order to keep the statement of Proposition \ref{prop:ProjIm} as clean as possible, we will need to adopt the convention that $(\ZZ/2\ZZ)^2$ is a dihedral group.
\end{example}

\begin{example}
	Let $p$ be a prime. 
	The goal of this example is to show that $\SL_2(\ZZ/p\ZZ)$ is not nilpotent. 
	A simple computation shows that $Z(\SL_2(\ZZ/p\ZZ)) = \langle -I \rangle$ and by definition $\SL_2(\ZZ/p\ZZ)/Z(\SL_2(\ZZ/p\ZZ))$ is $\PSL_2(\ZZ/p\ZZ)$. 
	A classical result \cite{JordanPSL} is that $\PSL_2(\ZZ/p\ZZ)$ is simple for all $p\geq5$. 
	Since the center of a group is always normal and $\PSL_2(\ZZ/p\ZZ)$ is clearly non-abelian, it must be that $Z(\PSL_2(\ZZ/p\ZZ))$ is trivial. Thus, $\PSL_2(\ZZ/p\ZZ)$ is not nilpotent. 
	This together with Proposition \ref{prop:NilpIsClosed} shows that $\SL_2(\ZZ/p\ZZ)$ is not nilpotent when $p\geq 5$. 
	The cases when $p=2$ and $p=3$ can be easily checked by hand. 
\end{example}

Before proceeding, we remind the reader of the following definition:
\begin{definition}
	Let $K/k$ be a Galois extension of fields. We say that $K/k$ is a {\it \bfseries nilpotent extension} if $\Gal({K}/k)$ is a nilpotent group. When the base field $k$ is clear from context, we will just say that $K$ is a {\bfseries nilpotent field} for brevity.
\end{definition}

\subsection{Classification of nilpotent division fields of prime level}
Step one in the process of determining when an elliptic curve $E/\QQ$ can have a nilpotent $n$-division field, is determining when the $p$-division fields can be nilpotent extensions of $\QQ$. Proposition \ref{prop:NilpIsClosed} tells us that if $\QQ(E[n])/\QQ$ is nilpotent, then $\QQ(E[d])/\QQ$ is nilpotent for all
$d \mid n$. Moreover, if $n = p_{1}^{a_{1}} \cdots p_{k}^{a_{k}}$
is the prime factorization of $n$, then $\QQ(E[n])/\QQ$ is nilpotent if and only if $\QQ(E[p_{i}^{a_{i}}])/\QQ$ is nilpotent for all $i$. To see why this is true, one only needs to recall the Galois correspondence as well as the fact that nilpotency is preserved under subgroups, quotients, and direct products.

For this reason, we start by studying $\QQ(E[p])/\QQ$ and use that information to understand what happens at level $p^2$ and further up the $p$-adic tower.

To that end, we need a way to divide up the subgroups of $\GL_2(\ZZ/p\ZZ)$ so that we can study them. Fortunately, Proposition \ref{prop:ProjIm} gives us exactly what we need.

\begin{remark}
\label{rem:nilprojim}
If $G$ is a subgroup of $\GL_{2}(\ZZ/p\ZZ)$, then $G$ is nilpotent if and only if
its image in $\PGL_{2}(\ZZ/p\ZZ)$ is nilpotent. The reason is as follows. If $G$ is nilpotent,
then its image in $\PGL_{2}(\ZZ/p\ZZ)$ is a quotient of $G$ and is hence nilpotent.
Conversely, if $G \leq \GL_{2}(\ZZ/p\ZZ)$ and the image of $G$ in $\PGL_{2}(\ZZ/p\ZZ)$
is nilpotent, then the image of $G$ in $\PGL_{2}(\ZZ/p\ZZ)$ is the quotient $G/N$,
where $N$ is the set of scalar multiples of the identity in $G$. This subgroup $N \leq Z(G)$
and by Theorem 5.13 of \cite{ConradSubSeriesI}, it follows that $G$ is nilpotent, since
$N$ can be extended into a central series.
\end{remark}

The following result gives a classification of when an admissible subgroup of $\GL_{2}(\ZZ/p\ZZ)$ is nilpotent.

\begin{proposition}\label{prop:ProjImIsDi}
Let $p$ be a prime and let $G$ be an admissible subgroup of 
$\GL_2(\ZZ/p\ZZ)$. Then $G$ is nilpotent if and only if $G$ is abelian,
or the image of $G$ in $\PGL_2(\ZZ/p\ZZ)$ is isomorphic to $D_{2^k}$ for some $k\geq2$. Further, if $p$ is odd, then $G$ is either contained in the normalizer of a split or non-split Cartan subgroup of $\GL_2(\ZZ/p\ZZ)$.
\end{proposition}
\begin{proof}
If $G$ is abelian, then it must be nilpotent. Remark~\ref{rem:nilprojim} shows that if the image
of $G$ in $\PGL_{2}(\ZZ/p\ZZ)$ is isomorphic to $D_{2^{k}}$ (a $2$-group) then $G$ is nilpotent.

Now we assume that $G$ is nilpotent and consider the cases based on Proposition~\ref{prop:ProjIm}. First, if $G$ contains $\SL_2(\ZZ/p\ZZ)$, then $G$ cannot be nilpotent since $\SL_2(\ZZ/p\ZZ)$ is not nilpotent.

Next, suppose that $G$ is an admissible nilpotent group such that $p \mid |G|$, and $G$ is contained in a Borel subgroup of $\GL_2(\ZZ/p\ZZ)$. Conjugating $G$ if necessary, we may assume $G$ is contained in the set of upper triangular matrices. The set of upper triangular matrices contains
a unique subgroup of order $p$, namely the cyclic group generated by
\[
A = \begin{pmatrix}
1&1\\0&1
\end{pmatrix}.
\]
Since we assumed that $G$ was nilpotent, we have from part (3) of Theorem \ref{thm:Nilp} every element in $G$ must commute with $A.$ 
Let \[
B = \begin{pmatrix}
	a&b\\0&d
\end{pmatrix}
\]
be an arbitrary element of $G$. Next we compute 

\[
AB = \begin{pmatrix}
a&b+d\\0&d
\end{pmatrix} \text{ and }
BA = \begin{pmatrix}
a & a+b\\
0 &d
\end{pmatrix}.
\]
From this we get that the only way that $AB = BA$ is if $a = d$. This implies that every element of $G$ must have square determinant, and $\det(G) = (\ZZ/p\ZZ)^\times$ now forces $p = 2$ and
\[
G = \left\langle \begin{pmatrix} 1&1\\0&1 \end{pmatrix}\right\rangle\subseteq \GL_2(\ZZ/2\ZZ),
\]
which is abelian.

We now consider case (1) of Proposition~\ref{prop:ProjIm}. In this case, $G$ is contained in a Cartan subgroup and is hence abelian.

In case (2) of Proposition~\ref{prop:ProjIm}, the projective image is dihedral and by Example~\ref{ex:WhenDnIsNilp}, a dihedral group is nilpotent if and only if its order is a power of $2$.

In case (3) of Proposition~\ref{prop:ProjIm} the projective image is isomorphic to $A_{4}$, $S_{4}$ or $A_{5}$. If $G$ were nilpotent this would imply that one of $A_{4}$, $S_{4}$ or $A_{5}$ is nilpotent. Proposition~\ref{prop:TrivialCenter} shows that none of these are nilpotent, since all have trivial center.
\end{proof}

We are now ready to apply what is known about the corresponding modular curves. 

\subsection{Modular curves associated to split Cartan subgroups}

In this section, we survey what is known about the modular curves associated to split Cartan subgroups of $\GL_2(\ZZ/p\ZZ)$. We let $C_s^+(p)$ be the normalizer of a split Cartan subgroup of $\GL_2(\ZZ/p\ZZ)$ and $X_s^+(p)$ the corresponding modular curve.

The work of Bilu, Parent, and Rebolledo in \cite{BiluParentRebolledo} gives an almost complete picture of rational points on $X_s^+(p)$. This work together with the work of Balakrishnan, Dogra, M\"{u}ller, Tuitman and Vonk in \cite{Balakrishnan} gives, among other things, the following theorem.

\begin{theorem}\label{thm:BPR}{\rm \cite{Balakrishnan,BiluParentRebolledo}}
If $p\geq 11$ is a prime, then the $\QQ$-rational points on $X_s^+(p)$ are cusps or correspond to elliptic curves with complex multiplication. 
\end{theorem}

Using this result we prove the following.

\begin{proposition}\label{prop:SplitMeans23or5}
Let $E/\QQ$ be an elliptic curve without complex multiplication and let $p$ be a prime such that $\Im\bar\rho_{E,p}$ is contained in $C_s^+(p)$. If $\QQ(E[p])/\QQ$ is a nilpotent extension, then $p \in \{ 2, 3, 5 \}$.
\end{proposition}
\begin{proof}
By Theorem~\ref{thm:BPR}, it suffices to rule out the case that $p = 7$. When $p = 7$, the projective image of $C_s^+(p)$ has size $2(7-1) = 12$ and so is not nilpotent. As shown in \cite{RSZB}, there are three maximal admissible subgroups of $C_{s}^{+}(7)$, and for each of these, the corresponding modular curve has genus $1$. For two of these, there are no non-CM points, while for the third of these, there is a non-CM point with $j = \frac{3^{3} \cdot 5 \cdot 7^{5}}{2^{7}}$. An elliptic curve with this $j$-invariant has mod $7$ image isomorphic to either $\Z/6\Z \times S_{3}$ or $\Z/3\Z \times S_{3}$ and neither of these groups is nilpotent.
\end{proof}

\subsection{Modular curves associated to non-split Cartan subgroups}

In this section, we survey what is known about the modular curves associated to non-split Cartan subgroups of $\GL_2(\ZZ/p\ZZ)$.
To start, we will let $p$ be an odd prime and we define the non-split Cartan subgroup of $\GL_2(\ZZ/p\ZZ)$ to be 
\[
C_{ns}(p) := \left\{\begin{pmatrix} a&\epsilon b\\ b & a \end{pmatrix}\colon a,b\in\ZZ/p\ZZ\hbox{ and } (a,b) \neq (0,0)\right\}
\]
where $\epsilon$ is a generator of $(\ZZ/p\ZZ)^\times$. The normalizer of this group in $\GL_2(\ZZ/p\ZZ)$ is exactly 
\[
C_{ns}^+(p) := \left\langle C_{ns}(p), \begin{pmatrix} 1& 0\\ 0 & -1 \end{pmatrix} \right\rangle.
\]
We will denote the modular curves corresponding to these groups by $X_{ns}(p)$ and $X_{ns}^+(p)$respectively. 

Much less is known about the rational points on the modular curves associated with the normalizers of the non-split Cartan subgroups compared to what is know about the split Cartan cases. 
These modular curves have some arithmetic properties that make analysis of their rational points particularly challenging. 
In particular, the Jacobians of these modular curves always have analytic rank at least as big as the genus of the curve. 
This rules out the traditional method of Chabauty and Coleman and requires more advanced techniques (which have been successful in two cases: see
\cite{Balakrishnan} and \cite{Balakrishnan2}).
Fortunately for us, enough is known that we will be able to say quite a bit about the situation unconditionally, and the remainder of what we need is covered by Conjecture \ref{conj:Uniformity}.

Below we give a summary of some of the relevant theorems for these modular curves. 

\begin{proposition}
Let $E/\QQ$ be an elliptic curve that does not have complex multiplication and let $p \geq 7$ be a prime such that $\Im\bar\rho_{E,p}$ is contained in $C_{ns}^{+}(p)$. Then $\Im\bar\rho_{E,p} = C_{ns}^{+}(p)$.
\end{proposition}
\begin{proof}
In \cite{Zyw} the stated result is proven for $p = 7$ (Theorem 1.5) and $p = 11$ (Theorem 1.6).
For $p = 13$, \cite{Balakrishnan} shows that there are no non-CM elliptic curves for which
$\Im\bar\rho_{E,p}$ is contained in $C_{ns}^{+}(p)$. The cases that $p \geq 17$ are handled
by combining Proposition 1.13 of \cite{Zyw} with Theorem 1.6 of \cite{FurioLombardo}.
\end{proof}

From this we can see that if $E/\QQ$ is an elliptic curve and $p\geq 7$ is a prime such that $\Im\bar\rho_{E,p}$ is conjugate to a subgroup of $C_{ns}^+(p),$ then $\Im\bar\rho_{E,p} = C_{ns}^+(p)$ and the image of $\bar\rho_{E,p}$ in $\PGL_2(\ZZ/p\ZZ)$ is a dihedral group with order 
\[
\frac{|C_{ns}^+(p)|}{(p-1)} = \frac{2(p^2-1)}{(p-1)} = 2(p+1).
\]
Combining this with Example \ref{ex:WhenDnIsNilp} and Remark~\ref{rem:nilprojim}, we get that in this case $\Im\bar\rho_{E,p}$ is nilpotent exactly when $2(p+1)$ is a power of 2 which can happen only when $p+1$ is a power of 2 or $p$ is a Mersenne prime. 

We summarize the discussion up to this point in the following proposition.

\begin{proposition}
\label{prop:nonCMnonsplit}	
Let $E/\QQ$ be an elliptic curve without complex multiplication and let $p$ be a prime such that $\Im\bar\rho_{E,p}$ is conjugate to a subgroup of $C_{ns}^+(p)$ and $\QQ(E[p])/\QQ$ is a nilpotent extension. Then $p$ is a Mersenne prime. 
\end{proposition}

This will be as much as we can say unconditionally. 
What we really need here is something like Theorem \ref{thm:BPR}, but for $X_{ns}^+(p)$. 
This is exactly why we need Conjecture \ref{conj:Uniformity}.

\begin{proposition}
	Let $E/\QQ$ be an elliptic curve without complex multiplication and let $p$ be a Mersenne prime. 
	Assuming Conjecture \ref{conj:Uniformity}, if $\Im\bar\rho_{E,p}$ is conjugate to a subgroup of $C_{ns}^+(p)$, then $p = 3$ or $p = 7$.
\end{proposition}

Before moving on to the case where $E/\QQ$ has complex multiplication, we state a proposition summarizing this section.

\begin{prop}\label{prop:PrimeLevelSummary}
Let $E/\QQ$ be an elliptic curve without complex multiplication and let $p$ be a prime such that $\QQ(E[p])/\QQ$ is a nilpotent extension. Then Conjecture \ref{conj:Uniformity} implies that $p\in\{2,3,5,7\}.$
\end{prop}

In Table~\ref{tab:Models} we give models for all modular curves of prime level
$p \in \{2, 3, 5, 7\}$ for which $\QQ(E[p])/\QQ$ is nilpotent. These modular curves are isomorphic to $\PP^{1}$ and hence there are infinitely many rational $j$-invariants of elliptic curves $E/\Q$ for which $\QQ(E[p])/\QQ$ is nilpotent.

\subsection{The case of complex multiplication}

One of the interesting properties of elliptic curves with complex multiplication is that their mod $p$ representations almost always have their images in the normalizer of a Cartan subgroup. Proposition 1.14(i) and (ii) of \cite{Zyw} state the following.

\begin{prop}
Let $E/\QQ$ be an elliptic curve with complex multiplication by an order $\mathcal{O}\neq\ZZ[\zeta_3]$ of a quadratic imaginary field $K$. 
Next, let $p\geq 3$ be a prime such that $p\nmid {\rm disc}(\mathcal{O})$. 
Then, $\Im\bar\rho_{E,p}$ is conjugate to
\[\begin{cases}
C_{s}^+(p) & \hbox{if $p\mathcal{O}_K$ splits in $\mathcal{O}_K$, and} \\
C_{ns}^+(p)& \hbox{if $p\mathcal{O}_K$ is inert in $\mathcal{O}_K$} .
\end{cases}
\]
\end{prop}

The point of this proposition is that in these cases the mod $p$ images is as large as possible. This is useful because we know that in these cases the image of $\Im\bar\rho_{E,p}$ of in $\PGL_2(\ZZ/p\ZZ)$ is dihedral by Proposition~\ref{prop:ProjIm} and has easily computable size. Remark~\ref{rem:nilprojim} and Example~\ref{ex:WhenDnIsNilp} now imply that $\Im\bar\rho_{E,p}$ is nilpotent
exactly when its image in $\PGL_{2}(\ZZ/p\ZZ)$ is a $2$-group.

Since $|C_s^+(p)| = 2(p-1)^2$ while $|C_{ns}^+(p)| = 2(p^2-1)$,
we have that $\Im\bar\rho_{E,p}$ is nilpotent
if it equals $C_{s}^{+}(p)$ and $p$ is a Fermat prime, or
if it equals $C_{ns}^{+}(p)$ and $p$ is a Mersenne prime.

% \begin{prop}\label{prop:CMLevelP}
% Let $E/\QQ$ be an elliptic curve with complex multiplication by $\mathcal{O}\neq\ZZ[\zeta_3]$ an order of $K$. 
% Let $p\geq 7$ be a prime such that $p\nmid {\rm disc}(\mathcal{O})$ and $\QQ(E[p])/\QQ$ is a nilpotent extension. 
% Then, either
% \begin{enumerate}
% 	\item the ideal $p\mathcal{O}_K$ splits in $\mathcal{O}_K$ and there exists a $k\geq 1$ such that $p = 2^k+1$, or
% 	\item the ideal $p\mathcal{O}_K$ is inert in $\mathcal{O}_K$ and there exists a $k\geq 1$ such that $p = 2^k-1$.
% \end{enumerate}
% \end{prop}

% Clearly, understanding how primes split in the ring of integers of quadratic fields is going to be central to understanding this situation. A classical result in algebraic number theory is that an odd prime $p\in \ZZ$ splits in the maximal order of $\QQ(\sqrt{a})$ if and only if $a$ is a square mod $p$. The interested reader is encouraged to see \cite{Cox} for more information. (\jeremy{Do we use this anywhere?})

The following result summarizes the situation and handles the cases that $p \mid {\rm disc}(\mathcal{O})$.
\begin{proposition}\label{prop:FermatOrMersenne}
Let $E/\QQ$ be an elliptic curve with complex multiplication by $\mathcal{O}\neq \ZZ[\zeta_3]$ and $p$ an odd prime. Then, $\QQ(E[p])/\QQ$ is nilpotent if and only if either $p$ splits in $\mathcal{O}$ and $p$ is a Fermat prime or $p$ is inert in $\mathcal{O}$ and $p$ is a Mersenne prime.
\end{proposition}
\begin{proof}
The discussion proceeding this theorem covers the case where $p\nmid {\rm disc}(\mathcal{O}).$ To handle the case where $p\mid{\rm disc}(\mathcal{O}),$ we refer to \cite[Theorem 1.14]{Zyw} which shows that in this case $\Im\bar\rho_{E,p}$ isomorphic to one of the following groups
\begin{align*}
G := \left\{ \begin{pmatrix}a&b\\0&\pm a \end{pmatrix}\colon a\in(\ZZ/p\ZZ)^\times, b\in\ZZ/p\ZZ \right\},\\
H_1 := \left\{ \begin{pmatrix}a&b\\0&\pm a \end{pmatrix}\colon a\in((\ZZ/p\ZZ)^\times)^2, b\in\ZZ/p\ZZ \right\}, \hbox{ or }\\
H_2 := \left\{ \begin{pmatrix}\pm a&b\\0& a \end{pmatrix}\colon a\in((\ZZ/p\ZZ)^\times)^2, b\in\ZZ/p\ZZ \right\}.
\end{align*}
In all three cases, the matrix $\begin{pmatrix} 1&1\\0&1\end{pmatrix}$ is in $\Im\bar\rho_{E,p}$ and so $\Im\bar\rho_{E,p}$ is not nilpotent by Proposition \ref{prop:ProjImIsDi}.
\end{proof}

\subsubsection{The case when $E$ has complex multiplication by $\ZZ[\zeta_3]$}

If $E/\QQ$ has complex multiplication by $\mathcal{O} = \ZZ[\zeta_3]$ we know that $j(E) = 0.$ Given such an elliptic curve, we know that there is always a $d\in\QQ^\times$ such that $E$ is isomorphic to the curve 
\[
E_d\colon y^2 = x^3+d.
\]
The images of the mod $p$ representations of $E$ depend on the value of $d$ modulo 6th powers. This relationship is explicitly classified in \cite[Propositions 1.15 and 1.16]{Zyw}. We summarize the relevant parts of those propositions here for the convenience of the reader.

%\jeremy{The part about $\Im \bar\rho_{E,3}$ was misquoted.}
\begin{thm}{\rm \cite[Propositions 1.15 and 1.16]{Zyw}}\label{thm:j=0}
Let $E/\QQ$ be an elliptic curve with complex multiplication by $\ZZ[\zeta_3]$. Then the curve $E$ can be given by a Weierstrass equation of the form 
\[
y^2 = x^3+d
\]
for some $d\in\QQ^\times.$
\begin{enumerate}
	\item If $d$ is a cube, then $\Im\bar\rho_{E,2} = \left\langle \begin{pmatrix} 1&1\\0&1 \end{pmatrix} \right\rangle$. Otherwise, $\Im\bar\rho_{E,2} = \GL_2(\ZZ/2\ZZ)$. 
	\item If $4d$ is not a cube, then $\Im\bar\rho_{E,3}$ is conjugate to 
\[
	\begin{cases} 
	\left\{ \begin{pmatrix} \pm 1 & a \\ 0 & b \end{pmatrix}\colon a\in \ZZ/3\ZZ \hbox{ and } b\in(\ZZ/3\ZZ)^\times \right\} & \text{ if neither } d \text{ nor } -3d \text{ is a square }\\
	\left\{ \begin{pmatrix} 1 & a\\0&b \end{pmatrix}\colon a\in (\ZZ/3\ZZ)^\times\hbox{ and } b\in\ZZ/3\ZZ \right\} & \text{ if } d \text{ is a square }\\
    \left\{ \begin{pmatrix} a & b\\ 0 & 1 \end{pmatrix}\colon a\in (\ZZ/3\ZZ)^\times\hbox{ and } b\in\ZZ/3\ZZ \right\} & \text{ if } 3d \text{ is a square. }
\end{cases}
\]
On the other hand, if $4d$ is a cube, then $\Im\bar\rho_{E,3}$ is conjugate to
\[
\begin{cases}
   \left\{ \begin{pmatrix} a & 0 \\ 0 & b \end{pmatrix} \colon a, b \in (\ZZ/3\ZZ)^{\times} \right\} & \text{ if neither } d \text{ nor } -3d \text{ is a square }\\
   \left\{ \begin{pmatrix} 1 & 0 \\ 0 & b \end{pmatrix} \colon b \in (\ZZ/3\ZZ)^{\times} \right\} & \text{ if either } d \text{ or } -3d \text{ is a square. }
\end{cases}
\]
	\item If $p\equiv 1 \bmod 9$, then $\Im\bar\rho_{E,p}$ is conjugate to $C_s^+(p).$
	\item If $p\equiv 8 \bmod 9$, then $\Im\bar\rho_{E,p}$ is conjugate to $C_{ns}^+(p).$
	\item Suppose that $p\equiv 4$ or $7\bmod 9$ and $e\in \{1,2\}$ such that $e\equiv \frac{p-1}{3} \bmod 3$. If $d\not\equiv 16p^e\bmod (\QQ^\times)^3$, then $\Im\bar\rho_{E,p}$ is conjugate to $C_s^+(p)$. If $d\equiv 16p^e\bmod (\QQ^\times)^3$, then $\Im\bar\rho_{E,p}$ is conjugate in $\GL_2(\ZZ/p\ZZ)$ to the subgroup of $C_s^+(p)$ consisting of matrices of the form 
	\[
		\begin{pmatrix} a&0\\0&b \end{pmatrix}\hbox{ and }\begin{pmatrix} 0&a\\b&0 \end{pmatrix}
	\]
	with $a,b\in (\ZZ/p\ZZ)^\times$ such that $a/b$ is a cube.

	\item Suppose that $p\equiv 2$ or $5\bmod 9$ and let $e\in \{1,2\}$ such that $-e\equiv \frac{p+1}{3} \bmod 3$. If $d\not\equiv 16p^e\bmod (\QQ^\times)^3$, then $\Im\bar\rho_{E,p}$ is conjugate to $C_{ns}^+(p)$. 
	If $d\equiv 16p^e\bmod (\QQ^\times)^3$, then $\Im\bar\rho_{E,p}$ is conjugate in $\GL_2(\ZZ/p\ZZ)$ to the subgroup generated by the unique index 3 subgroup of $C_{ns}(p)$ and $\begin{pmatrix} 1&0\\0&-1 \end{pmatrix}.$
\end{enumerate}
\end{thm}

The take away from this theorem is that for $E_{d} \colon y^{2} = x^{3} + d$, whether or not $\Im\bar\rho_{E,p}$ is nilpotent is completely controlled by $d\bmod (\QQ^\times)^3.$

For example, condition (1) tells us that $\QQ(E_d[2])/\QQ$ is a nilpotent extension exactly when $d$ is a cube. Similarly, condition (2) says that $\QQ(E_d[3])/\QQ$ is nilpotent when $4d$ is a cube. 

We note that there are no Fermat primes $\equiv 1 \pmod{9}$. By Remark~\ref{rem:nilprojim}
we have that $\Im\bar\rho_{E,p}$ is nilpotent if and only if its image in $\PGL_{2}(\FF_{p})$, namely a dihedral group of order $2(p-1)$, is nilpotent. By Example~\ref{ex:WhenDnIsNilp} this cannot occur since $2(p-1)$ is not a power of $2$.

Thus if $p \equiv 1 \pmod{9}$
the image of $\Im\bar\rho_{E,p}$ in $\PGL_{2}(\ZZ/p\ZZ)$ is dihedral if and only if 

so condition (3) never yields a nilpotent $\Q(E[p])/\Q$. Likewise, there are no Mersenne primes $p \equiv 8 \pmod{9}$.

The last cases that we have to deal with are the special cases that arise in cases (5) and (6) of Theorem \ref{thm:j=0}. In cases (5) and (6) respectively, the image of $\bar\rho_{E,p}$ is contained in an index 3 subgroup of $C_s^+(p)$ and $C_{ns}^+(p)$ respectively. If we are in condition (5), then image of $\Im\bar\rho_{E,p}$ inside of $\PGL_2(\ZZ/p\ZZ)$ is a dihedral group of size $\frac{2(p-1)}{3}$, while in condition (6) the image of $\Im\bar\rho_{E,p}$ inside of $\PGL_2(\ZZ/p\ZZ)$ is a dihedral group of size $\frac{2(p+1)}{3}.$ This along with our previous analysis gives the following proposition.

\begin{prop}\label{prop:j=0primelevel}
Let $E_d\colon y^2 = x^3 + d$ and $p$ a prime. Then $\QQ(E_d[p])/\QQ$ is nilpotent if and only if
\[
\begin{cases}
d \equiv 1 \pmod{(\QQ^{\times})^{3}} \text{ if } p = 2,\\
d \equiv 2 \pmod{(\QQ^{\times})^{3}} \text{ if } p = 3,\\
d \equiv 2\cdot p^{\frac{p-1}{3}} \pmod{(\QQ^{\times})^{3}} \text{ if }
p = 3 \cdot 2^{k} + 1\hbox{ for some $k\geq 1$},\\
d \equiv 2\cdot p^{\frac{p+1}{3}} \pmod{(\QQ^{\times})^{3}} \text{ if }
p = 3 \cdot 2^{k} - 1\hbox{ for some $k\geq 1$}.\\
\end{cases}
\]
\end{prop}
%\jeremy{If we change Theorem 1.7 to express the condition in terms of $d$, should we change this proposition too?}

\begin{example}
	Let $E$ be the elliptic curve given by
	\[
y^2 = x^3 + 16\cdot97^2.
	\]
	We check in the LMFDB that the image of $\bar\rho_{E,97}$ is conjugate to the group with RSZB label \texttt{97.14259.1103.1}. One can check directly that this group is nilpotent and so $\QQ(E[97])/\QQ$ is nilpotent. 
\end{example}

\begin{remark}
If $E$ is an elliptic curve with complex multiplication by $\mathcal{O} = \ZZ[\zeta_3]$, then Proposition~\ref{prop:j=0primelevel} shows that $\QQ(E[p])/\QQ$ is nilpotent for at most one prime $p$. 
\end{remark}

Now we prove Corollary~\ref{cor:no19}.
\begin{proof}
If $E/\Q$ is an elliptic curve for which $\Q(E[19])/\Q$ is nilpotent then by Proposition~\ref{prop:ProjImIsDi}, $\Im\bar\rho_{E,p}$ is contained in the normalizer of a split or non-split Cartan subgroup. If $E$ is non-CM, the split Cartan case cannot occur by Proposition~\ref{prop:SplitMeans23or5} and the non-split Cartan case cannot occur by Proposition~\ref{prop:nonCMnonsplit}. If $E$ has CM by an order $\mathcal{O}$, then Proposition~\ref{prop:FermatOrMersenne} forces $\mathcal{O} = \Z[\zeta_{3}]$. However, by Proposition~\ref{prop:j=0primelevel} only cases where $p = 3 \cdot 2^{k} \pm 1$ can occur and $19$ does not have this form.

Proposition~\ref{prop:j=0primelevel} allows us to construct elliptic curves $E_{d} \colon y^{2} = x^{3} + d$ for which $\Q(E_{d}[p])/\Q$ is nilpotent for
$p = 2$, $p = 3$, $p = 5 = 3 \cdot 2 - 1$, $p = 7 = 3 \cdot 2 + 1$, $p = 11 = 3 \cdot 2^{2} - 1$
and $p = 13 = 3 \cdot 2^{2} + 1$. The prime $p = 17$ is a Fermat prime which splits
in $\Z[i]$ and so if $E : y^{2} = x^{3} - x$, then $\Q(E[17])/\Q$ is nilpotent by Proposition~\ref{prop:FermatOrMersenne}.
\end{proof}

% In Table \ref{tab:PrimeLevel} we list the modular curves that we need to consider and determine which of them have rational points. 

% \begin{center}
% \begin{table}[h]
% \renewcommand{\arraystretch}{1.5}
% \begin{tabular}{|c|c|c|c|}\hline
% $p$ & RSZB Label & Common Label & Has Rational Points \\\hline
% \multirow{2}{*}{2}& \texttt{2.2.0.1} & $X_{ns}(2)$ & Yes \\
%  & \texttt{2.3.0.1} & $X_{0}(2)$ & Yes \\\hline
% \multirow{1}{*}{3}& \texttt{3.3.0.1} & $X_{ns}^+(3)$ & Yes \\\hline
% \multirow{2}{*}{5}& \texttt{5.15.0.1} & $X_s^+(5)$ & Yes \\
%  & \texttt{5.20.0.2} & $X_{ns}(5)$ & No \\\hline
% \multirow{2}{*}{7}& \texttt{7.21.0.1} & $X_{ns}^+(7)$ & Yes \\
%  & \texttt{7.56.1.1} & $X_s(7)$ & No \\\hline

% \end{tabular}
% \caption{Maximal admissible subgroups of $\GL_2(\ZZ/p\ZZ)$ to be considered assuming Conjecture \ref{conj:Uniformity}.}\label{tab:PrimeLevel}
% \end{table}
% \end{center}

\section{Nilpotent groups of prime-power level}\label{sec:PrimePowerLevel}

Suppose that $p$ is a prime and that $E/\QQ$ is an elliptic curve such that $\QQ(E[p^k])/\QQ$ is a nilpotent extension for some $k\geq 2$. 
The first observation we make is that since nilpotency is closed under quotients, Proposition \ref{prop:NilpIsClosed}, we know that $\QQ(E[p^i])/\QQ$ is a nilpotent extension for all $1\leq i \leq k,$ in particular, and this would mean $\QQ(E[p])/\QQ$ is a nilpotent extension. 
As usual, we will have to handle the case when $p=2$ separately, but thanks to Proposition \ref{prop:ProjImIsDi}, when $p$ is odd, we only have to deal with the case that $\Im\bar\rho_{E,p}$ is contained in either the normalizer of a split or non-split Cartan subgroup of $\GL_2(\ZZ/p\ZZ).$ 

\subsection{The case when $p=2$}

In this case there are exactly two ways that $\QQ(E[2])/\QQ$ can be nilpotent. 
In order for $\QQ(E[2])/\QQ$ to be nilpotent, either $E$ can have a square discriminant or $E$ can have a point of order 2 defined over $\QQ$.

We start this case by considering what the image of $\bar\rho_{E,4}$ could be if we know that $E$ has square discriminant and
\[
\Im\bar\rho_{E,2} \subseteq \left\langle\begin{pmatrix} 1&1\\1&0 \end{pmatrix}\right\rangle.
\] 
Let $\pi_2\colon \GL_2(\ZZ/4\ZZ) \to \GL_2(\ZZ/2\ZZ)$ be the standard component-wise reduction map and let
\[
G_2 = \left\langle\begin{pmatrix} 1&1\\1&0 \end{pmatrix}\right\rangle\hbox{ and }
G_4 := \pi_2^{-1}(G_2).
\]
The next step is to search for admissible nilpotent subgroups of $G_4$ up to conjugation with the additional property that their image mod $2$ is exactly {equal to} $G_2$. Such a subgroup must have
order which is a multiple of $6$ (a factor of $3$ coming from the image in $G_{2}$
and a factor of $2$ coming from the determinant being surjective). There are two such groups up to conjugacy: one with order $6$ and one with order $12$, with
the former contained in the latter. Neither of these groups are admissible. Computing conjugacy classes of the order $12$ subgroup shows that there are three conjugacy classes of elements of order $2$ and none of these elements fix an element of $(\ZZ/4\ZZ)^{2}$ of order $4$, which the image of complex conjugation under $\bar\rho_{E,4}$ must. (Code for this calculation can be found in the file {\tt subsec51.m} at \cite{code}.) Thus in the case when $\Im\bar\rho_{E,2}$ is conjugate to $G_2$, there is no way that $\QQ(E[4])/\QQ$ can be a nilpotent extension. 

The next case is when $E/\QQ$ has a point of order 2 defined over $\QQ$. In this case, $\QQ(E[2])/\QQ$ is either a quadratic extension or trivial. 
Letting $\pi_2\colon \GL_2(\ZZ/2^k\ZZ)\to \GL_2(\ZZ/2\ZZ),$ we have that $|\ker(\pi_2)| = 2^{4(k-1)}$. 
From this we have that \[
|\pi_2^{-1}(G_2)| = 2^{4(k-1)}|G_2|
\] and in particular $\pi_2^{-1}(G_2)$ is a 2-group. From 
Theorem~\ref{thm:Nilp}, we know that $\pi_2^{-1}(G_2)$ is \emph{always} nilpotent. The upshot of this is that if $E/\QQ$ is an elliptic curve with a point of order two defined over $\QQ$, then for every $k\geq 1$, $\QQ(E[2^k])/\QQ$ is a nilpotent extension.

\begin{proposition}\label{prop:PowersOf2}
	Let $E/\QQ$ be an elliptic curve such that $\QQ(E[2])/\QQ$ is a nilpotent extension. 
	Then, either the discriminant of $E$ is a square, in which case $\QQ(E[2^k])/\QQ$ is not nilpotent for any $k\geq 2$, or $E$ has a rational point of order $2$, in which case $\QQ(E[2^k])/\QQ$ is nilpotent for all $k\geq 1$.
\end{proposition}

%\jeremy{Major changes were made to the remainder of this section.}

\subsection{The case when $p$ is odd}

\begin{proposition}\label{prop:Oddprimesquared}
Suppose that $G$ is a nilpotent subgroup of $\GL_{2}(\ZZ/p^{2}\ZZ)$ and let
$\pi \colon G \to \GL_{2}(\FF_{p})$ be the reduction mod $p$ map. Assume that $p \nmid |\pi(G)|$. Then at least one of the following is true:
\begin{enumerate}
\item $G$ is abelian.
\item $\ker(\pi) \subseteq \{ \alpha I : \alpha \in (\ZZ/p^{2} \ZZ)^{\times} \text{ with } \alpha \equiv 1 \pmod{p} \}$.
\end{enumerate}
\end{proposition}

\begin{proof}
Let $P$ be a Sylow $p$-subgroup of $G$. Since $|\pi(G)|$ has order coprime to $p$, we have that $\pi(P) = \{ 1 \}$. In particular, $P$ is contained in the set of matrices $\equiv I \pmod{p}$. The set of matrices $\equiv I \pmod{p}$ is an abelian subgroup of $\GL_{2}(\ZZ/p^{2} \ZZ)$ order $p^{4}$. In particular $P$ is abelian and $\ker \pi = P$. Let 
\[
H = \prod_{\substack{Q \in \Syl_{q}(G) \\ q \ne p}} Q
\]
be a complement of $P$ in $G$. Note that $\pi(G) = \pi(PH) = \pi(P) \pi(H) = \{ 1 \} \cdot \pi(H) = \pi(H)$ and also $H \cap \ker \pi \subseteq H \cap P = \{ 1 \}$. Thus $\pi \colon H \to \pi(G)$ is an isomorphism.

Case I: There exists an element of $P$ that is not a scalar multiple of the identity.

This implies that there is some $X \in M_{2}(\FF_{p})$ so that $I + pX \in P$ and $X$ is not a scalar multiple of the identity. If $Y \in \pi(G)$, there is some $\tilde{Y} \in H$ so that $\pi(\tilde{Y}) = Y$. The assumption that $G$ is nilpotent implies that $(I + pX)$ must commute with $\tilde{Y}$, and this implies that $XY = YX$ in $M_{2}(\FF_{p})$. The assumption on $X$ implies that $X$ is a {\bf cyclic matrix}. This is a matrix $X$ whose minimal polynomial and characteristic polynomial are the same. 

Corollary 4.4.18 of \cite{HornJohnson} implies that for every cyclic matrix $X$, its centralizer in $M_{2}(\FF_{p})$ is equal to $\FF_{p}[X]$, the set of all polynomials in $X$ with coefficients in $\FF_{p}$. This is a commutative subring of $M_{2}(\FF_{p})$, and this implies that $\pi(G) \subseteq \FF_{p}[X]^{\times}$ is abelian. Since $H \simeq \pi(G)$, it follows that $H$ is abelian. Since $G \simeq P \times H$, it follows that $G$ is abelian.

Case II: Every element of $P$ is a scalar multiple of the identity.

Since $P = \ker(\pi)$, in this case, condition (2) is clearly true.
\end{proof}

As a consequence of this result, we can establish the following result.

\begin{prop}\label{prop:MainP2}
Let $E/\QQ$ be an elliptic curve and let $p$ be an odd prime. Then $\QQ(E[p^2])/\QQ$ is not a nilpotent extension. 
\end{prop}
\begin{proof}
Assume that $E/\QQ$ is an elliptic curve, $p$ is an odd prime, and $\QQ(E[p^{2}])/\QQ$ is nilpotent. Let $G = \Im\bar\rho_{E,p^{2}}$. If we are in case (1) of Proposition~\ref{prop:Oddprimesquared}, then $\QQ(E[p^{2}])/\QQ$ is abelian, which contradicts the main result of \cite{AbelDivFields}. If we are in case (2) of Proposition~\ref{prop:Oddprimesquared} and $\pi \colon G \to \GL_{2}(\FF_{p})$ is the reduction mod $p$ map, then $\ker(\pi) \cap (G \cap \SL_{2}(\ZZ/p^{2} \ZZ)) = 1$ and this implies that we have a near coincidence of level $(p^{2},p)$. By Proposition~\ref{prop:nearsummary}, we must have that $p = 3$ and $E$ corresponds to a rational point on the curve with RSZB label \texttt{9.27.0.1}. However, the main result of \cite{RSZB} implies that for such an elliptic curve, the mod $9$ image of Galois must equal \texttt{9.27.0.1}, which is not nilpotent. 
\end{proof}

As a consequence, if $E/\QQ$ is an elliptic curve, $p$ is an odd prime, and $n \geq 2$, $\QQ(E[p^{n}])/\QQ$ is not a nilpotent extension.

\begin{comment}
We are now ready to complete the proof Theorem \ref{thm:main_CongCurve}.

\begin{proof}[Proof of Theorem \ref{thm:main_CongCurve}]
A classical result in the study of constructible number is that $\alpha\in\CC$ is constructible if and only if the degree of the Galois closure of $\QQ(\alpha)$, over $\QQ$ is a power of two. In particular, if we assume that $\alpha$ is constuctible and let $K$ be the Galois closure of $\QQ(\alpha)$, then it must be that $K/\QQ$ is nilpotent. Examining the results from the previous sections, we see that if 
\[
E\colon y^2 = x^3-x,
\]
then $E$ has complex multiplication by $\ZZ[i].$ In this case, $\QQ(E[2^k])/\QQ$ is always nilpotent and has Galois group a power of. Similarly, if $p$ is a Fermat prime, then $\QQ(E[p])/\QQ$ is not only nilpotent, but it is also an extension of 2-powered degree. On the other hand, when $p$ is Mersenne prime, $\QQ(E[p])/\QQ$ is nilpotent, but it is not an extension of 2-powered degree. In this case, using \cite[Propositions 1.15 and 1.16]{Zyw} we see that $[\QQ(E[p]):\QQ] = 3\cdot2^m$ for some non-negative integer $m$. While this does not prevent these extensions from being nilpotent, it does prevent them from being constructible. 
\end{proof}
\end{comment}

\section{Nilpotent groups of composite level}\label{sec:CompositeLevel}

In this section, we will complete the proof of Theorem~\ref{thm:main_Nilp}.
Since $\QQ(E[n])$ is the composite of $\QQ(E[p^{k}])$ for every prime power factor $p^{k}$ of $n$, we have that $\QQ(E[n])/\QQ$ is nilpotent if and only if every $\QQ(E[p^{k}])/\QQ$ is nilpotent. From Section~\ref{sec:PrimePowerLevel}, this only occurs if $k = 1$, or $p = 2$ and $E$ has a rational point of order $2$.

First, we consider the case that $E/\QQ$ is an elliptic curve with complex multiplication. From Section~\ref{sec:PrimeLevel}, we have a classification of when $\QQ(E[p])/\QQ$ is nilpotent based on the mod $p$ image of Galois for $E$, which is determined by whether $p$ splits, is inert, or ramifies in the CM field. From Section~\ref{sec:PrimePowerLevel}, we have that $\QQ(E[p^{k}])/\QQ$ is nilpotent for $k \geq 2$ if and only if $p = 2$ and $E$ has a rational point of order $2$. From these results, the complex multiplication cases of Theorem~\ref{thm:main_Nilp} follow.

Next we consider the case that $E/\QQ$ is an elliptic curve without complex multiplication under the assumption of Conjecture~\ref{conj:Uniformity}.
In this case, Section~\ref{sec:PrimeLevel} implies that if $\QQ(E[p])/\QQ$ is nilpotent then either $p \in \{ 2, 5 \}$ or $p$ is a Mersenne prime and $\Im \bar\rho_{E,p}$ is contained in the normalizer of the non-split Cartan subgroup of $\GL_{2}(\FF_{p})$, which entails that $p \in \{ 2, 3, 5, 7 \}$.
All that remains is for us to determine which combinations of the possible nilpotent mod $p$ images can occur simultaneously. 
Given possible mod $p$ and mod $q$ images of Galois $G$ and $H$, we construct the fiber product $X_{G} \times_{X_{0}(1)} X_{H}$. This is the curve given by $\pi_G(x) = \pi_H(y)$. 
%\jeremy{I removed the reference to the modular curves being genus zero. You can still take the fiber product of modular curves of genus $\geq 1$ in the same way.} 
The results of this computation are listed in the following table.

\begin{center}
\begin{table}[h]
\renewcommand{\arraystretch}{1.5}
\begin{tabular}{|c|c|c|c|c|}\hline
$(p,q)$ & $\Im\bar\rho_{E,p}$ & $\Im\bar\rho_{E,q}$ & $\Im\bar\rho_{E,pq}$ & Has Non-Cuspidal Rational Points \\\hline
\multirow{2}{*}{(2,3)}& \texttt{2.2.0.1} & \texttt{3.3.0.1} & \texttt{6.6.1.1} & No - Genus 1 Rank 0 \\ %$y^2 = x^3-27$
 & \texttt{2.3.0.1} & \texttt{3.3.0.1} & \texttt{6.9.0.1} & Yes \\\hline

\multirow{2}{*}{(2,5)}& \texttt{2.2.0.1} & \texttt{5.15.0.1} & \texttt{10.30.2.2} & No - Genus 2 Rank 0 \\ %$ y^{2} = x^{5} - 8 x^{4} + x^{3} + 23 x^{2} - 6 x - 11$
 & \texttt{2.3.0.1} & \texttt{5.15.0.1} & \texttt{10.45.1.1}& No - Genus 1 Rank 0 \\\hline%$y^{2} + \left(x + 1\right) y= x^{3} + x^{2} - 3 x + 1$

\multirow{2}{*}{(2,7)}& \texttt{2.2.0.1} & \texttt{7.21.0.1} & \texttt{14.42.3.1} & No - Genus 3 Rank 0 \\%$y^2 = 8 x^{8} - 28 x^{7} + 7 x^{6} + 49 x^{5} - 14 x^{4} - 21 x^{3} - 28 x^{2} + 16 x$

 & \texttt{2.3.0.1} & \texttt{7.21.0.1} & \texttt{14.63.2.1} & No - Genus 2 Rank 0 \\\hline

\multirow{1}{*}{(3,5)}& \texttt{3.3.0.1} & \texttt{5.15.0.1} & \texttt{15.45.1.1} & Yes \\\hline
\multirow{1}{*}{(3,7)}& \texttt{3.3.0.1} & \texttt{7.21.0.1} & \texttt{21.63.1.1} & Yes\\\hline
\multirow{1}{*}{(5,7)}& \texttt{5.15.0.1} & \texttt{7.21.0.1} & \texttt{35.315.19.1} & No - Genus 19 Analytic Rank 15\\\hline

\end{tabular}
\caption{Potential composite level nilpotent images}\label{tab:CompositeLevel}
\end{table}
\end{center}

The curves that are genus 1 or 2 with rank zero can be handled using standard techniques. The rank 0 genus 3 curve was shown to have no non-cuspidal rational points corresponding to elliptic curves without complex multiplication in \cite{JacksonCompImage}. This leaves us with with one remaining curve, with label
\texttt{35.315.19.1}. Theorem A.7 of Appendix~A to \cite{RSZB} shows that 
if $X_{H}$ is a modular curve of level $N$, every simple factor of the Jacobian of $X_{H}$ is isogenous to a simple factor of $J_{1}(N^{2})$, and Section~6 of \cite{RSZB} explains how this information can be used to determine the decomposition of the Jacobian of $X_{H}$. This decomposition is recorded in the beta version of the LMFDB
at \href{https://beta.lmfdb.org/ModularCurve/Q/35.315.19.a.1/}{the link here}. In particular, $X_{H}$ factors up to isogeny as the product of nine $\Q$-simple abelian varieties (of dimensions $2$, $3$ and $4$).

Work of Kolyvagin-Logachev \cite{KolyLoga} shows that if a simple abelian variety $A$ of $GL_{2}$-type has analytic rank $0$ it must have rank $0$, while if it has
analytic rank $r$ equal to its dimension, then it must also have algebraic rank $r$. (This latter result is not stated by Kolyvagin and Logachev, but their ideas suffice to prove it. For more detail about how this follows, see \cite[Section~7]{DograLeFourn}.) This hypothesis is easy to verify since $L(A,s)$ factors as a product of
$\dim A$ modular $L$-functions. From this, it follows that the analytic and algebraic
rank of the Jacobian of \texttt{35.315.19.1} are both $15$. In theory, this curve could be attacked using the method of Chabauty and Coleman since the genus higher than the rank, but computing on a curve of genus 19 is rather unwieldy. However, the techniques
of Lemos \cite{LemosSomeCases} apply to this situation. 

The first theorem proved by Lemos in \cite{LemosSomeCases} states the following.
\begin{theorem}{\rm \cite[Theorem 1.4]{LemosSomeCases}}\label{thm:Lemos}
	Let $E/\QQ$ be an elliptic curve without complex multiplication. Suppose that there exists a prime $q$ for which $\Im \bar\rho_{E,q}$ is contained in a subgroup of $C_s^+(q)$. Then $\bar\rho_{E,p}$ is surjective for all $p>37.$
\end{theorem}

The work in Lemos's paper does not, however, require full strength of the assumption that $p > 37$. We wish to explain why the work in Lemos's paper implies the following result.

\begin{theorem}
\label{thm:LemosExtend}
Let $E/\QQ$ be an elliptic curve without complex multiplication. Suppose that there exist a prime $q$ for which $\Im \bar\rho_{E,q}$ is contained in a subgroup of $C_s^+(q)$. Then $\bar\rho_{E,p}$ is not contained in $C_{ns}^{+}(p)$ for any $p > 3$
with $p \ne q$.
\end{theorem}
Setting $p=7$ and $q=5$ above implies that the only rational points on the modular curve \texttt{35.315.19.1} are cusps or CM points.

\begin{proof}[Proof of Theorem~\ref{thm:LemosExtend}]
As this result is really contained in \cite{LemosSomeCases} we will give an overview of the steps in Lemos's argument highlighting the necessary hypotheses on $p$ and $q$.
First, Theorem~\ref{thm:BPR} implies that $q \in \{ 2, 3, 5, 7 \}$. As a consequence, $X_{0}(q)$ has genus zero.

The modular curve $X_{s}(q)$ parametrizes elliptic curves with two independent cyclic $q$-isogenies. In Lemma 3.4 of \cite{LemosSomeCases}, Lemos shows that there is an isomorphism
\[
  \theta \colon X_{s}(q) \times_{X_{0}(1)} X_{ns}^{+}(p) \to X_{0}(q^{2}) \times_{X_{0}(1)} X_{ns}^{+}(p)
\]
and that $\theta$ commutes with natural involutions on the source and the target. (The involution on the source interchanges the kernels of the two isogenies, and the involution on the target is the Atkin-Lehner involution $w_{q^{2}}$.) Lemos then defines
a map $g \colon X_{0}(q^{2}) \times_{X_{0}(1)} X_{ns}^{+}(p) \to J(X_{0}(q) \times_{X_{0}(1)} X_{ns}^{+}(p))$ by taking a point $P$ and mapping it to the difference of its images under the two different degeneracy maps (coming from the two covers $X_{0}(q^{2}) \to X_{0}(q)$).

Work of Darmon and Merel \cite[Proposition 7.1]{DarmonMerel} shows that there is a projection $\pi \colon J(X_{0}(q) \times_{X_{0}(1)} X_{ns}^{+}(p)) \to A$ to a positive-dimensional abelian variety with rank $0$ for which the kernel of $\pi$ is connected and is stable under the action of Hecke operators. (For their application, Darmon and Merel only need this for $q = 2$ or $q = 3$, but as they indicate, this part of the argument applies to any $q$ with $p \nmid q$.) 

Define $h = \pi \circ g$ and let $\mathcal{O}$ be the ring of integers in $\Q(\zeta_{p} + \zeta_{p}^{-1})$ and $R = \mathcal{O}[1/(2qp)]$. Roughly speaking, Proposition 3.6 of \cite{LemosSomeCases} shows that $h$ is a formal immersion at $\infty$ for every prime ideal $\mathfrak{p}$ of $R$. The proof of this proposition is standard and the key assumption is that $X_{0}(q)$ has genus $0$. It follows from this that $\theta \circ h \colon X_{s}(q) \times_{X_{0}(1)} X_{ns}^{+}(p) \to A$ is a formal immersion at $\infty$ for every prime ideal $\mathfrak{p}$ of $R$. 

Now, if $P$ is a rational point on $X_{s}(q) \times_{X_{0}(1)} X_{ns}^{+}(p)$ corresponding to an elliptic curve $E$ which has potential multiplicative reduction at some prime $\ell \not\in \{ 2, p, q \}$, then it meets one of the cusps at the fiber at $\ell$. Without loss of generality this cusp can be chosen to be infinity and this implies that $\theta \circ h(P) = Q \in A(\Q)$ reduces to zero in $\tilde{A}(\FF_{\ell})$. (Here $\tilde{A}$ is the special fiber of the N\'eron model of $A$ over $\Z_{\ell}$.)
This implies that $Q = 0$ and the fact that $\theta \circ h$ is a formal immersion
and $\ell > 2$ implies that $P = \infty$. (One way to make this conclusion is using
Proposition 2.4 of \cite{LF16}.)

It follows from this that if $E/\Q$ is an elliptic curve with mod $p$ image contained in $C_{ns}^{+}(p)$ and mod $q$ image contained in $C_{s}^{+}(q)$, then
$E$ cannot have potentially multiplicative reduction at any prime $\ell \not\in \{2, p, q\}$. The primes of potentially multiplicative reduction are precisely those that
divide the denominator of $j(E)$. As a consequence $j(E) \in \Z[1/(2pq)]$. More is true however. Lemos notes in \cite[Proposition 3.3]{LemosSomeCases} that the assumption that the mod $p$ image of Galois is contained in $C_{ns}^{+}(p)$ implies that $E$ has potentially supersingular reduction at $p$, and that if $E$ has potentially multiplicative reduction at $\ell \ne p$, then $\ell \equiv \pm 1 \pmod{p}$. (These same observations were made earlier by Zywina.) 
It follows that the only primes that can divide the denominator of $j(E)$ are those
that are $\equiv \pm 1 \pmod{p}$. Since $p > 3$, none of $2$, $3$, $5$ or $7$
(the only options for $q$) can be $\equiv \pm 1 \pmod{p}$. It follows
that $j(E) \in \Z$. 

Lemos proceeds to show using an explicit isomorphism $X_{s}^{+}(q) \simeq \mathbb{P}^{1}$ that there are only finite integral $j$-invariants
of elliptic curves $E/\Q$ with mod $q$ image of Galois contained in $C_{s}^{+}(q)$. For $q \in \{ 3, 5, 7 \}$ these are explicitly listed on \cite[p. 749]{LemosSomeCases}. They all have CM except for $j = -5000$ ($q = 5$) and $j=-1728$ ($q=3$). For elliptic curves with these $j$-invariants, the image of $\rho_{E,p}$ is not contained
in the normalizer of a non-split Cartan subgroup for any $p > 3$. The
case of $q = 2$ was previously handled in \cite[p. 142]{Lemos2}. Here there are $25$ integral $j$-invariants of elliptic curves $E$ with image in $C_{s}^{+}(2)$ (which is a Borel subgroup of $\GL_{2}(\FF_{2})$). Of these $25$, there are $18$ are non-CM $j$-invariants. For elliptic curves with these
$j$-invariants, the LMFDB indicates that the image of $\rho_{E,p}$ is contained
in $C_{ns}^{+}(p)$ for some $p$ only for $p = 3$ and $j \in \{ -64, 4913, 238328, 16974593 \}$.
\end{proof}

In Table~\ref{tab:Models}, we give models for the modular curves from Table~\ref{tab:CompositeLevel} that have non-cuspidal rational points.

\begin{center}
\begin{table}[h]
\renewcommand{\arraystretch}{2.2}
\begin{tabular}{|c|c|c|}\hline
$G$ & $X_G$ & $\pi_G\colon X_G\to \PP_\QQ^1$\\\hline
\texttt{ 2.2.0.1 }& $\PP^1$ & $f_2(t) = t^2+1728$ \\\hline
\texttt{ 2.3.0.1 }& $\PP^1$ & $h_2(t) = \frac{(256 - t)^3}{t^2}$ \\\hline
\texttt{ 3.3.0.1 }& $\PP^1$ & $f_3(t) = t^3$ \\\hline
\texttt{ 5.15.0.1 }& $\PP^1$ & $ f_5(t) = \frac{(t+5)^3(t^2-5)^3(t^2+5t+10)^3}{(t^2+5t+5)^5}$ \\\hline
\texttt{ 6.9.0.1 }& $\PP^1$ & $ f_6(t) = \frac{(t^3+3t^2+3t-15)^3}{(t+1)^3}$ \\\hline
\texttt{ 7.21.0.1 }& $\PP^1$ & $ f_7(t) = \frac{(2t-1)^3(t^2-t+2)^3(2t^2+5t+4)^3(5t^2+2t-4)^3}{(t^3+2t^2-t-1)^7} $
\\\hline
\texttt{ 15.45.1.1 }& $y^2+y = x^3+1$ & $ f_{15}(x,y) = \frac{(y+3)^{3}(y^{2}-4y-1)^{3}(y^{2}+y+4)^{3}}{x^6(y^{2}+y-1)^{5}}$\\\hline
\texttt{ 21.63.1.1 }& $ y^{2} + y= x^{3} + 12$ & $f_{21}(x,y) = f_7\left( \frac{x^2 + 5x - 14}{x^2 - 4x + 3y + 19} \right)$\\\hline
\end{tabular}
\caption{Models of the modular curves relevant for Theorem~\ref{thm:main_Nilp}. The models for curves of prime level come from \cite{SZ}. The remaining genus 0 curves can be computed as fiber products of the curves of prime level. The models for the genus 1 modular curves of composite level are computed as the fiber product of the prime level modular curves that are computed in \cite{SZ} as well.}\label{tab:Models}
\end{table}
\end{center}

Working without the assumption of Conjecture~\ref{conj:Uniformity}, we
must consider the possibility that there is an elliptic curve $E/\Q$ for which $\Im\bar\rho_{E,5}$ is contained in the normalizer of the split Cartan mod $5$ and for which $\Im \bar\rho_{E,p}$ is contained in the normalizer of a non-split Cartan modulo $p$ for some Mersenne prime $p$. Theorem~\ref{thm:LemosExtend} above shows this is only possible for $p = 3$, and the elliptic curves for which this occurs are parametrized by the modular curve with label \texttt{15.45.1.1}. 

In addition, we must consider the possibility that there is an elliptic curve $E/\Q$ for which $\im \bar\rho_{E,p}$ is contained in the normalizer of a non-split Cartan modulo $p$ for some Mersenne prime $p$, and which also has a rational point of order $2$. This implies the mod $2$ image of Galois is contained in a Borel subgroup,
which for $p=2$ is equal to $C_{s}^{+}(2)$. The desired result again follows from
Theorem~\ref{thm:LemosExtend}. This concludes the proof of Theorem~\ref{thm:main_Nilp}.

\bibliographystyle{plain} 
\bibliography{bib}

\end{document}